\documentclass[11pt,a4paper]{article}
\usepackage[top=25mm,bottom=35mm,left=20mm,right=20mm,truedimen]{geometry}

\usepackage{amsmath,newtxtext,newtxmath,bm}

\usepackage{amsthm}
\usepackage{mathtools}
\mathtoolsset{showonlyrefs=true}
\usepackage[dvipdfmx]{graphicx}
\usepackage[subrefformat=parens]{subcaption}

\allowdisplaybreaks

\newtheorem{thm}{Theorem}[section]
\newtheorem{lmm}[thm]{Lemma}

\newtheorem{cor}[thm]{Corollary}

\usepackage{color}

\newcommand{\im}{\mathrm{i}}        
\newcommand{\rmvec}{\mathrm{vec}}   
\newcommand{\rmtridiag}{\mathrm{tridiag}}
\newcommand{\mqty}[1]{\begin{bmatrix}#1\end{bmatrix}}
\newcommand{\mdet}[1]{\begin{vmatrix}#1\end{vmatrix}}

\title{
  Matrix equation representation of the convolution equation and its unique solvability
  }
\author{
  Yuki Satake\thanks{Department of Apllied Physics, Graduate School of Engineering, Nagoya University, Furo-cho, Chikusa-ku, Nagoya 464-8603, Japan, Email: \texttt{\{y-satake,sogabe,kemmochi,zhang\}@na.nuap.nagoya-u.ac.jp}} \and
  Tomohiro Sogabe$^\ast$ \and
  Tomoya Kemmochi$^\ast$, \and
  Shao-Liang Zhang$^\ast$
}
\date{}

\begin{document}
\maketitle
\begin{abstract}
  We consider the convolution equation~$F*X=B$, where $F\in\mathbb{R}^{3\times 3}$ and $B\in\mathbb{R}^{m\times n}$ are given, and $X\in\mathbb{R}^{m\times n}$ is to be determined.
  The convolution equation can be regarded as a linear system with a coefficient matrix of special structure.
  This fact has led to many studies including efficient numerical algorithms for solving the convolution equation.
  In this study, we show that the convolution equation can be represented as a generalized Sylvester equation.
  Furthermore, for some realistic examples arising from image processing, we show that the generalized Sylvester equation can be reduced to a simpler form, and analyze the unique solvability of the convolution equation.
\end{abstract}
\section{Introduction}
  We consider the following convolution equation
  \begin{equation}
    F*X=B, \label{eq:convolution}
  \end{equation}
  where $F=[f_{ij}]\in\mathbb{R}^{3\times 3}$ and $B=[b_{ij}]\in\mathbb{R}^{m\times n}$ are given, and $X=[x_{ij}]\in\mathbb{R}^{m\times n}$ is to be determined.
  The symbol $*$ represents the convolution operator, i.e., the left-hand side of the $(i,j)$ entry of~\eqref{eq:convolution} is defined by
  \begin{equation}
    [F*X]_{ij}:=\sum^3_{l_1=1}\sum^3_{l_2=1} f_{l_1l_2}x_{i-l_1+2, j-l_2+2}. \label{def:convolution}
  \end{equation}
  Since the above expression~\eqref{def:convolution} is only defined for cases when $2\le i\le m-1$ and $2\le j\le n-1$,
  boundary conditions to define $b_{ij}$ for cases when $i\in\{1,m\}$ or $j\in\{1,n\}$ are needed.
  In other words, the values of $x_{0j}, x_{m+1,j}, x_{i0},x_{i,n+1}$ need to be set artificially.
  In this paper, we consider the following three boundary conditions.
  \begin{itemize}
    \item Zero boundary condition: \\
    \begin{equation}
      x_{0j}=x_{m+1,j}=x_{i0}=x_{i,n+1}=0.\label{bc:zero}
    \end{equation}
    \item Periodic boundary condition: \\
    \begin{equation}
      x_{0j}=x_{mj},\quad x_{m+1,j}=x_{1j},\quad x_{i0}=x_{in},\quad x_{i,n+1}=x_{i1}.\label{bc:periodic}
    \end{equation}
    \item Reflexive boundary condition: \\
    \begin{equation}
      x_{0j}=x_{1j},\quad x_{m+1,j}=x_{mj},\quad x_{i0}=x_{i1},\quad x_{i,n+1}=x_{in}.\label{bc:reflexive}
    \end{equation}
  \end{itemize}
  
  The convolution equation~\eqref{eq:convolution} appears in image restoration problems~\cite{afonso2010,hansen_deconvolution_2002,thevenaz2000} arising from astronomy~\cite{kundur1998,starck2002} and medical imaging~\cite{michailovich2005}.
  In image restoration problems, $F$, $X$, and $B$ of~\eqref{eq:convolution} represent a filter matrix, an original image, and an observed degraded image, respectively.

  Eq.~\eqref{def:convolution} under one of~\eqref{bc:zero},~\eqref{bc:periodic},~\eqref{bc:reflexive} can be rewritten as the following linear system
  \begin{equation}
    \mathcal{F}\bm{x}=\bm{b}, \label{eq:ls}
  \end{equation}
  where $\bm{x}:=\rmvec(X)\in\mathbb{R}^{mn}$ and $\bm{b}:=\rmvec(B)\in\mathbb{R}^{mn}$, and $\rmvec$ denotes a vectorization operator (e.g.,~\cite{zhan2013}).
  The coefficient matrix $\mathcal{F}\in\mathbb{R}^{mn\times mn}$ has a special structure that depends on boundary conditions.
  Under the zero boundary condition, $\mathcal{F}$ has a block Toeplitz with Toeplitz blocks (BTTB) structure.
  Under the periodic boundary condition, $\mathcal{F}$ has a block circulant with circulant blocks (BCCB) structure.
  Under the reflexive boundary condition, $\mathcal{F}$ is a sum of BTTB, a Block Toeplitz with Hankel blocks (BTHB), a Block Hankel with Toeplitz blocks (BHTB), and a Block Hankel with Hankel blocks (BHHB), i.e., a Block-Toeplitz-plus-Hankel with Toeplitz-plus-Hankel-Blocks (BTHTHB) matrix.

  Many algorithms for solving~\eqref{eq:convolution} are based on the vectorized form~\eqref{eq:ls}.
  Although the coefficient matrix~$\mathcal{F}$ can be large, the above structures allow the solution of~\eqref{eq:ls} to be computed efficiently without explicitly constructing~$\mathcal{F}$.
  For example, the following algorithms have been proposed:
  algorithms for the inversion of BTTB matrix~\cite{wax1983}, numerical algorithms and their preconditioners for solving BTTB systems~\cite{alonso2005, chan1994, sun2001, yagle2001}, algorithms for solving BCCB systems by using the fast Fourier transform (FFT)~\cite{chen1987,rjasanow1994}, and  numerical algorithms for solving BTHTHB systems by using discrete cosine transform (DCT)~\cite{ng1999a}. 

  As described above, representing the convolution equation~\eqref{eq:convolution} as the linear system~\eqref{eq:ls} leads to prosperous results (including numerical algorithms).
  This motivates us to find a different representation.
  In this study, we provide the following representation: the convolution equation~\eqref{eq:convolution} can be transformed into a generalized Sylvester equation by using special matrices, which is our main contribution.
  This study may allow one to find mathematical features and numerical solvers of the convolution equation~\eqref{eq:convolution} by means of procedures for the generalized Sylvester equation.
  Furthermore, the generalized Sylvester equation can be reduced to simpler forms for some filter matrices that appear in image processing.
  Using this, we show the necessary and sufficient conditions for the unique solvability of the convolution equation~\eqref{eq:convolution} with the filter matrices. 

  The rest of this paper is organized as follows.
  In Section~\ref{sec2}, we show that the convolution equation~\eqref{eq:convolution} can be rewritten as a generalized Sylvester equation by using some special matrices.
  In Section~\ref{sec3}, the existence of unique solutions of~\eqref{eq:convolution} with some specific filter matrices related image processing is discussed.
  The concluding remarks are provided in Section~\ref{sec4}.
  
  Throughout this paper, $I_n\in\mathbb{R}^{n\times n}$ and $\mathbb{N}:=\{1,2,3,\ldots\}$ denote the $n\times n$ identity matrix and the set of all natural numbers, respectively.
\section{Rewriting the convolution equation to a generalized matrix equation\label{sec2}}
  In this section, we show that the convolution equation~\eqref{eq:convolution} with the zero, periodic or reflexive boundary condition can be rewritten as a generalized Sylvester equation.

  Before describing the main results, let us recall some special matrices.
  We first introduce shift matrices.
  The $n\times n$ shift matrices $U_n$ and $L_n$ are defined by
  \begin{equation}
    U_n:=\mqty{
      0 & 1 & 0 & \cdots & 0 \\
      \vdots & \ddots & \ddots & \ddots & \vdots \\
      \vdots & \ddots & \ddots & \ddots & 0 \\
      \vdots & \ddots & \ddots & \ddots & 1 \\
      0 & \cdots & \cdots & \cdots & 0}\in\mathbb{R}^{n\times n}, \qquad
    L_n:=\mqty{
      0 & \cdots & \cdots & \cdots & 0 \\
      1 & \ddots & \ddots & \ddots & \vdots \\
      0 & \ddots & \ddots & \ddots & \vdots\\
      \vdots & \ddots & \ddots & \ddots & \vdots \\
      0 & \cdots & 0 & 1 & 0}\in\mathbb{R}^{n\times n}.\label{shiftmatrix}
  \end{equation}
  The matrices $U_n$ and $L_n$ are called an upper shift matrix and a lower shift matrix respectively.
  The upper and lower shift matrices represent linear transformations that shift the components of column vectors one position up and down respectively, i.e., it follows that
  \begin{equation}
    U_m X=\mqty{
      x_{21} & x_{22} & \cdots & x_{2n} \\
      x_{31} & x_{32} & \cdots & x_{3n} \\
      \vdots & \vdots & \cdots & \vdots \\
      x_{m1} & x_{m2} & \cdots & x_{mn} \\
      0      & 0      & \cdots & 0
    }, \qquad
    L_m X=\mqty{
      0      & 0      & \cdots & 0      \\
      x_{11} & x_{12} & \cdots & x_{1n} \\
      x_{21} & x_{22} & \cdots & x_{2n} \\
      \vdots & \vdots & \cdots & \vdots \\
      x_{m-1,1} & x_{m-1,2} & \cdots & x_{m-1,n} 
    },\label{shift-x}
  \end{equation}
  where $X=[x_{ij}]\in\mathbb{R}^{m\times n}$.
  Similarly, it is clear that 
  \begin{equation}
    XL_n=\mqty{
      x_{12} & x_{13} & \cdots & x_{1n} & 0 \\
      x_{22} & x_{23} & \cdots & x_{2n} & 0 \\
      \vdots & \vdots & \vdots & \vdots & \vdots \\
      x_{m2} & x_{m3} & \cdots & x_{mn} & 0
    }, \qquad
    XU_n=\mqty{
      0      & x_{11} & x_{12} & \cdots & x_{1,n-1} \\
      0      & x_{21} & x_{22} & \cdots & x_{2,n-1} \\
      \vdots & \vdots & \vdots & \vdots & \vdots \\
      0      & x_{m1} & x_{m2} & \cdots & x_{m, n-1}
    }.\label{x-shift}
  \end{equation}
  From~\eqref{shift-x} and~\eqref{x-shift}, the following properties hold:
  \begin{align}
    \begin{aligned}
      U_m XL_n&=\mqty{
        x_{22} & \cdots & x_{2n} & 0 \\
        \vdots & \ddots & \vdots & \vdots \\
        x_{m2} & \cdots & x_{mn} & \vdots \\
        0      & \cdots & \cdots & 0
      },&
      U_m XU_n&=\mqty{
        0      & x_{21} & \cdots & x_{2,n-1} \\
        \vdots & \vdots & \ddots & \vdots \\
        \vdots & x_{m1} & \cdots & x_{m, n-1} \\
        0      & \cdots & \cdots & 0
      },\\
      L_m XL_n&=\mqty{
        0         & \cdots & \cdots & 0 \\
        x_{12}    & \cdots & x_{1n} & \vdots \\
        \vdots    & \ddots & \vdots & \vdots \\
        x_{m-1,2} & \cdots & x_{m-1,n} & 0
      },&
      L_m XU_n&=\mqty{
        0      & \cdots    & \cdots & 0 \\
        \vdots & x_{11}    & \cdots & x_{1,n-1} \\
        \vdots & \vdots    & \ddots & \vdots \\
        0      & x_{m-1,1} & \cdots & x_{m-1,n-1}
      }.\label{shift-x-shift}
    \end{aligned}
  \end{align}
  
  Next, we introduce cyclic shift matrices.
  The $n\times n$ cyclic shift matrices $U_n^{(\mathrm{P})}$ and $L_n^{(\mathrm{P})}$ are defined by
  \begin{equation}
    U_n^{(\mathrm{P})}:=\mqty{
      0 & 1 & 0 & \cdots & 0 \\
      \vdots & \ddots & \ddots & \ddots & \vdots \\
      \vdots & \ddots & \ddots & \ddots & 0 \\
      0      & \ddots & \ddots & \ddots & 1 \\
      1      & 0      & \cdots & \cdots & 0}\in\mathbb{R}^{n\times n}, \qquad
    L_n^{(\mathrm{P})}:=\mqty{
      0 & \cdots & \cdots & 0      & 1 \\
      1 & \ddots & \ddots & \ddots & 0 \\
      0 & \ddots & \ddots & \ddots & \vdots\\
      \vdots & \ddots & \ddots & \ddots & \vdots \\
      0 & \cdots & 0 & 1 & 0}\in\mathbb{R}^{n\times n}.\label{cshiftmatrix}
  \end{equation}
  A straightforward calculation yields
  \begin{align}
    \begin{aligned}
    U_m^{(\mathrm{P})} X&=\mqty{
      x_{21} & x_{22} & \cdots & x_{2n} \\
      x_{31} & x_{32} & \cdots & x_{3n} \\
      \vdots & \vdots & \cdots & \vdots \\
      x_{m1} & x_{m2} & \cdots & x_{mn} \\
      x_{11} & x_{12} & \cdots & x_{1n}
    }, &
    L_m^{(\mathrm{P})}X&=\mqty{
      x_{m1} & x_{m2} & \cdots & x_{mn}  \\
      x_{11} & x_{12} & \cdots & x_{1n} \\
      x_{21} & x_{22} & \cdots & x_{2n} \\
      \vdots & \vdots & \cdots & \vdots \\
      x_{m-1,1} & x_{m-1,2} & \cdots & x_{m-1,n} 
      },\\
    XL_n^{(\mathrm{P})}&=\mqty{
      x_{12} & x_{13} & \cdots & x_{1n} & x_{11} \\
      x_{22} & x_{23} & \cdots & x_{2n} & x_{21}\\
      \vdots & \vdots & \vdots & \vdots & \vdots \\
      x_{m2} & x_{m3} & \cdots & x_{mn} & x_{m1}
      }, &
    XU_n^{(\mathrm{P})}&=\mqty{
      x_{1n} & x_{11} & x_{12} & \cdots & x_{1,n-1} \\
      x_{2n}  & x_{21} & x_{22} & \cdots & x_{2,n-1} \\
      \vdots & \vdots & \vdots & \vdots & \vdots \\
      x_{mn} & x_{m1} & x_{m2} & \cdots & x_{m, n-1}
    }.\label{x-cshift}
    \end{aligned}
  \end{align}
  Similarly, we have
  \begin{align}
    \begin{aligned}
      U_m^{(\mathrm{P})} XL_n^{(\mathrm{P})}&=\mqty{
        x_{22} & \cdots & x_{2n} & x_{21} \\
        \vdots & \ddots & \vdots & \vdots \\
        x_{m2} & \cdots & x_{mn} & x_{m1} \\
        x_{12}      & \cdots & x_{1n} & x_{11}
      },&
      U_m^{(\mathrm{P})} X U_n^{(\mathrm{P})}&=\mqty{
        x_{2n} & x_{21} & \cdots & x_{2,n-1} \\
        \vdots & \vdots & \ddots & \vdots \\
        x_{mn} & x_{m1} & \cdots & x_{m, n-1} \\
        x_{1n} & x_{11} & \cdots & x_{1,n-1}
      },\\
      L_m^{(\mathrm{P})} XL_n^{(\mathrm{P})}&=\mqty{
        x_{m2}    & \cdots & x_{mn} & x_{m1} \\
        x_{12}    & \cdots & x_{1n} & x_{11} \\
        \vdots    & \ddots & \vdots & \vdots \\
        x_{m-1,2} & \cdots & x_{m-1,n} & x_{m-1, 1}
      },&
      L_m^{(\mathrm{P})}XU_n^{(\mathrm{P})}&=\mqty{
        x_{mn}    & x_{m1}    & \cdots & x_{m,n-1} \\
        x_{1n}    & x_{11}    & \cdots & x_{1,n-1} \\
        \vdots    & \vdots    & \ddots & \vdots \\
        x_{m-1,n} & x_{m-1,1} & \cdots & x_{m-1,n-1}
      }.\label{cshift-x-cshift}
    \end{aligned}
  \end{align}

  Here we define the following matrices
  \begin{equation}
    U_n^{(\mathrm{R})}:=\mqty{
      0      & 1      & 0      & \cdots & 0 \\
      0      & \ddots & \ddots & \ddots & \vdots \\
      \vdots & \ddots & \ddots & \ddots & 0 \\
      \vdots & \ddots & \ddots & 0      & 1 \\
      0      & \cdots & \cdots & 0      & 1}\in\mathbb{R}^{n\times n}, \qquad
    L_n^{(\mathrm{R})}:=\mqty{
      1      & 0      & \cdots & \cdots & 0 \\
      1      & 0      & \ddots & \ddots & \vdots \\
      0      & \ddots & \ddots & \ddots & \vdots \\
      \vdots & \ddots & \ddots & \ddots & 0 \\
      0      & \cdots & 0      & 1      & 0}\in\mathbb{R}^{n\times n}.\label{qshiftmatrix}
  \end{equation}
  For products of~\eqref{qshiftmatrix} and a matrix $X\in\mathbb{R}^{m\times n}$, it follows that
  \begin{align}
    \begin{aligned}
      U_m^{(\mathrm{R})} X&=\mqty{
      x_{21} & x_{22} & \cdots & x_{2n} \\
      x_{31} & x_{32} & \cdots & x_{3n} \\
      \vdots & \vdots & \cdots & \vdots \\
      x_{m1} & x_{m2} & \cdots & x_{mn} \\
      x_{m1} & x_{m2} & \cdots & x_{mn}
      }, &
      L_m^{(\mathrm{R})}X&=\mqty{
        x_{11} & x_{12} & \cdots & x_{1n} \\
        x_{11} & x_{12} & \cdots & x_{1n} \\
        x_{21} & x_{22} & \cdots & x_{2n} \\
        \vdots & \vdots & \cdots & \vdots \\
        x_{m-1,1} & x_{m-1,2} & \cdots & x_{m-1,n} 
        },\\
      X{U_n^{(\mathrm{R})}}^\top&=\mqty{
        x_{12} & x_{13} & \cdots & x_{1n} & x_{1n} \\
        x_{22} & x_{23} & \cdots & x_{2n} & x_{2n}\\
        \vdots & \vdots & \vdots & \vdots & \vdots \\
        x_{m2} & x_{m3} & \cdots & x_{mn} & x_{mn}
        }, &
      X{L_n^{(\mathrm{R})}}^\top&=\mqty{
        x_{11} & x_{11} & x_{12} & \cdots & x_{1,n-1} \\
        x_{21}  & x_{21} & x_{22} & \cdots & x_{2,n-1} \\
        \vdots & \vdots & \vdots & \vdots & \vdots \\
        x_{m1} & x_{m1} & x_{m2} & \cdots & x_{m, n-1}
      },\\
      U_m^{(\mathrm{R})} X{U_n^{(\mathrm{R})}}^\top&=\mqty{
        x_{22} & \cdots & x_{2n} & x_{2n} \\
        \vdots & \ddots & \vdots & \vdots \\
        x_{m2} & \cdots & x_{mn} & x_{mn} \\
        x_{m2} & \cdots & x_{mn} & x_{mn}
      },&
      U_m^{(\mathrm{R})} X {L_n^{(\mathrm{R})}}^\top&=\mqty{
        x_{21} & x_{21} & \cdots & x_{2,n-1} \\
        \vdots & \vdots & \ddots & \vdots \\
        x_{m1} & x_{m1} & \cdots & x_{m, n-1} \\
        x_{m1} & x_{m1} & \cdots & x_{m,n-1}
      },\\
      L_m^{(\mathrm{R})} X{U_n^{(\mathrm{R})}}^\top&=\mqty{
        x_{12}    & \cdots & x_{1n} & x_{1n} \\
        x_{12}    & \cdots & x_{1n} & x_{1n} \\
        \vdots    & \ddots & \vdots & \vdots \\
        x_{m-1,2} & \cdots & x_{m-1,n} & x_{m-1, n}
      },&
      L_m^{(\mathrm{R})}X{L_n^{(\mathrm{R})}}^\top&=\mqty{
        x_{11}    & x_{11}    & \cdots & x_{1,n-1} \\
        x_{11}    & x_{11}    & \cdots & x_{1,n-1} \\
        \vdots    & \vdots    & \ddots & \vdots \\
        x_{m-1,1} & x_{m-1,1} & \cdots & x_{m-1,n-1}
      }.\label{qshift-x-qshift}
    \end{aligned}
  \end{align}

  We can now state our main results.
  By using the shift matrices~\eqref{shiftmatrix}, the convolution equation with the zero boundary condition can be rewritten as a generalized Sylvester equation.
  \begin{thm}\label{thm:zero}
    {\bfseries(Zero boundary condition)}
    Let $F\in\mathbb{R}^{3\times 3}$, $X\in\mathbb{R}^{m\times n}$ and $B\in\mathbb{R}^{m\times n}$.
    Then, the convolution equation $F*X=B$ with the zero boundary condition is equivalent to a generalized Sylvester equation
    \begin{equation}
      F_1 XL_n + F_2 X + F_3 XU_n = B, \label{eq:gen-Sylvester}
    \end{equation}
    where 
    \begin{equation}
      F_i:=\rmtridiag(f_{1i},f_{2i}, f_{3i})=\mqty{
        f_{2i} & f_{1i} & & \\
        f_{3i} & f_{2i} & \ddots & \\
               & \ddots & \ddots & f_{1i} \\
               &        & f_{3i} & f_{2i}
      }\in\mathbb{R}^{m\times m}, \qquad i = 1, 2, 3,
    \end{equation}
    and $U_n, L_n\in\mathbb{R}^{n\times n}$ are shift matrices~\eqref{shiftmatrix}.
  \end{thm}
  \begin{proof}
    From~\eqref{def:convolution},~\eqref{bc:zero},~\eqref{shift-x},~\eqref{x-shift} and~\eqref{shift-x-shift}, it follows that
    \begin{align}
      (F*X)_{ij}=&
       f_{33}(L_mXU_n)_{ij} + f_{32}(L_mX)_{ij} + f_{31}(L_mXL_n)_{ij} + f_{23}(XU_n)_{ij} \\
       & \quad + f_{22}(X)_{ij}  + f_{21}(XL_n)_{ij} + f_{13}(U_mXU_n)_{ij} + f_{12}(U_mX)_{ij} + f_{11}(U_mXL_n)_{ij}=b_{ij},
    \end{align}
    which implies
    \begin{equation}
      \mqty{
        f_{21} & f_{11} &        & \\
        f_{31} & f_{21} & \ddots & \\
               & \ddots & \ddots & f_{11} \\
               &        & f_{31} & f_{21}
      }XL_n + 
      \mqty{
        f_{22} & f_{12} &        & \\
        f_{32} & f_{22} & \ddots & \\
               & \ddots & \ddots & f_{12} \\
               &        & f_{32} & f_{22}
      }X + 
      \mqty{
        f_{23} & f_{13} &        & \\
        f_{33} & f_{23} & \ddots & \\
               & \ddots & \ddots & f_{13} \\
               &        & f_{33} & f_{23}
      }XU_n = B.
    \end{equation}
    Hence, $F*X=B$ with the zero boundary condition is equivalent to the generalized Sylvester equation~\eqref{eq:gen-Sylvester}.
  \end{proof}

  By using the cyclic shift matrices~\eqref{cshiftmatrix}, the convolution equation with the periodic boundary condition can be rewritten as a generalized Sylvester equation.
  \begin{thm}\label{thm:periodic}
    {\bfseries(Periodic boundary condition)}
    Let $F\in\mathbb{R}^{3\times 3}$, $X\in\mathbb{R}^{m\times n}$ and $B\in\mathbb{R}^{m\times n}$.
    Then, the convolution equation $F*X=B$ with the periodic boundary condition is equivalent to a generalized Sylvester equation
    \begin{equation}
      \tilde{F}_1 XL_n^{(\mathrm{P})} + \tilde{F}_2 X + \tilde{F}_3 XU_n^{(\mathrm{P})} = B, \label{eq:gen-Sylvester:p}
    \end{equation}
    where 
    \begin{equation}
      \tilde{F}_i:=\rmtridiag(f_{1i},f_{2i}, f_{3i}) + 
      \mqty{
        0      & 0      & \cdots & 0      & f_{3i} \\
        0      & 0      & \cdots & 0      & 0      \\
        \vdots & \vdots & \ddots & \vdots & \vdots \\
        0      & 0      & \cdots & 0      & 0 \\
        f_{1i} & 0      & \cdots & 0      & 0 
      }
      \in\mathbb{R}^{m\times m}, \qquad i = 1, 2, 3,
    \end{equation}
    and $U_n^{(\mathrm{P})}, L_n^{(\mathrm{P})}\in\mathbb{R}^{n\times n}$ are cyclic shift matrices~\eqref{cshiftmatrix}.
  \end{thm}
  \begin{proof}
    From~\eqref{def:convolution},~\eqref{bc:periodic},~\eqref{x-cshift} and~\eqref{cshift-x-cshift}, it follows that
    \begin{align}
      (F*X)_{ij}=&
       f_{33}\left(L_m^{(\mathrm{P})}XU_n^{(\mathrm{P})}\right)_{ij} + f_{32}\left(L_m^{(\mathrm{P})}X\right)_{ij} + f_{31}\left(L_m^{(\mathrm{P})}XL_n^{(\mathrm{P})}\right)_{ij} + f_{23}\left(XU_n^{(\mathrm{P})}\right)_{ij} \\
       & \quad + f_{22}(X)_{ij}  + f_{21}\left(XL_n^{(\mathrm{P})}\right)_{ij} + f_{13}\left(U_m^{(\mathrm{P})}XU_n^{(\mathrm{P})}\right)_{ij} + f_{12}\left(U_m^{(\mathrm{P})}X\right)_{ij} + f_{11}\left(U_m^{(\mathrm{P})}XL_n^{(\mathrm{P})}\right)_{ij}=b_{ij},
    \end{align}
    which implies
    \begin{align}
      &\mqty{
        f_{21} & f_{11} & 0      & \cdots & 0      & f_{31} \\
        f_{31} & f_{21} & f_{11} & \ddots & \ddots & 0 \\
        0      & f_{31} & \ddots & \ddots & \ddots & \vdots \\
        \vdots & \ddots & \ddots & \ddots & \ddots & 0 \\
        0      & \ddots & \ddots & \ddots & \ddots & f_{11} \\
        f_{11} & 0      & \cdots & 0      & f_{31} & f_{21}
      }XL_n^{(\mathrm{P})}
      + \mqty{
        f_{22} & f_{12} & 0      & \cdots & 0      & f_{32} \\
        f_{32} & f_{22} & f_{12} & \ddots & \ddots & 0 \\
        0      & f_{32} & \ddots & \ddots & \ddots & \vdots \\
        \vdots & \ddots & \ddots & \ddots & \ddots & 0 \\
        0      & \ddots & \ddots & \ddots & \ddots & f_{12} \\
        f_{12} & 0      & \cdots & 0      & f_{32} & f_{22}
      }X \\
      &\quad + 
      \mqty{
        f_{23} & f_{13} & 0      & \cdots & 0      & f_{33} \\
        f_{33} & f_{23} & f_{13} & \ddots & \ddots & 0 \\
        0      & f_{33} & \ddots & \ddots & \ddots & \vdots \\
        \vdots & \ddots & \ddots & \ddots & \ddots & 0 \\
        0      & \ddots & \ddots & \ddots & \ddots & f_{13} \\
        f_{13} & 0      & \cdots & 0      & f_{33} & f_{23}
      }XU_n^{(\mathrm{P})}=B.
    \end{align}
    Hence, $F*X=B$ with the periodic boundary condition is equivalent to the generalized Sylvester equation~\eqref{eq:gen-Sylvester:p}.
  \end{proof}

  By using the tridiagonal matrices~\eqref{qshiftmatrix}, the convolution equation with the reflexive boundary condition can be rewritten as a generalized Sylvester equation.
  \begin{thm}\label{thm:reflexive}
    {\bfseries(Reflexive boundary condition)}
    Let $F\in\mathbb{R}^{3\times 3}$, $X\in\mathbb{R}^{m\times n}$ and $B\in\mathbb{R}^{m\times n}$.
    Then, the convolution equation $F*X=B$ with the reflexive boundary condition is equivalent to a generalized Sylvester equation
    \begin{equation}
      \hat{F}_1 X{U_n^{(\mathrm{R})}}^\top + \hat{F}_2 X + \hat{F}_3 X{L_n^{(\mathrm{R})}}^\top = B, \label{eq:gen-Sylvester:r}
    \end{equation}
    where 
    \begin{equation}
      \hat{F}_i:=\rmtridiag(f_{1i},f_{2i}, f_{3i}) + 
      \mqty{
        f_{3i} & 0      & \cdots & 0      & 0 \\
        0      & 0      & \cdots & 0      & 0      \\
        \vdots & \vdots & \ddots & \vdots & \vdots \\
        0      & 0      & \cdots & 0      & 0 \\
        0      & 0      & \cdots & 0      & f_{1i} 
      }
      \in\mathbb{R}^{m\times m}, \qquad i = 1, 2, 3,
    \end{equation}
    and $U_n^{(\mathrm{R})}, L_n^{(\mathrm{R})}\in\mathbb{R}^{n\times n}$ are defined by~\eqref{qshiftmatrix}.
  \end{thm}
  \begin{proof}
    From~\eqref{def:convolution},~\eqref{bc:reflexive} and~\eqref{qshift-x-qshift}, it follows that
    \begin{align}
      (F*X)_{ij}=&
      f_{33}\left(L_m^{(\mathrm{R})}X{L_n^{(\mathrm{R})}}^\top\right)_{ij} + f_{32}\left(L_m^{(\mathrm{R})}X\right)_{ij} + f_{31}\left(L_m^{(\mathrm{R})}X{U_n^{(\mathrm{R})}}^\top\right)_{ij} + f_{23}\left(X{L_n^{(\mathrm{R})}}^\top\right)_{ij} \\
      & \quad + f_{22}(X)_{ij}  + f_{21}\left(X{U_n^{(\mathrm{R})}}^\top\right)_{ij} + f_{13}\left(U_m^{(\mathrm{R})}X{L_n^{(\mathrm{R})}}^\top\right)_{ij} + f_{12}\left(U_m^{(\mathrm{R})}X\right)_{ij} + f_{11}\left(U_m^{(\mathrm{R})}X{U_n^{(\mathrm{R})}}^\top\right)_{ij}=b_{ij},
    \end{align}
    which implies
    \begin{align}
      &\mqty{
        f_{21} + f_{31} & f_{11} &   &     & \\
        f_{31} & f_{21} & \ddots & & \\
               & \ddots & \ddots & \ddots & \\
               &        & \ddots & f_{21} & f_{11} \\
               &        &        & f_{31} & f_{21} + f_{11}
      }X{U_n^{(\mathrm{R})}}^\top + 
      \mqty{
        f_{22} + f_{32} & f_{12} &   &     & \\
        f_{32} & f_{22} & \ddots & & \\
               & \ddots & \ddots & \ddots & \\
               &        & \ddots & f_{22} & f_{12} \\
               &        &        & f_{32} & f_{22} + f_{12}
      }X \\
      &\quad + 
      \mqty{
        f_{23} + f_{33} & f_{13} &   &     & \\
        f_{33} & f_{23} & \ddots & & \\
               & \ddots & \ddots & \ddots & \\
               &        & \ddots & f_{23} & f_{13} \\
               &        &        & f_{33} & f_{23} + f_{13}
      }X{L_n^{(\mathrm{R})}}^\top=B.
    \end{align}
    Hence, $F*X=B$ with the reflexive boundary condition is equivalent to the generalized Sylvester equation~\eqref{eq:gen-Sylvester:r}.
  \end{proof}
\section{Existence of unique solutions of the generalized Sylvester equations\label{sec3}}
  This section deals with some realistic filters~$F\in\mathbb{R}^{3\times 3}$ arising from image processing and discusses the uniqueness of solutions of the convolution equation~\eqref{eq:convolution} for each filter using our main results.
  Here, we consider the following filter matrices.
  \begin{enumerate}
    \item Box blur filter (BOX):\\
      \begin{equation}
        F_{\rm BOX}=\frac{1}{9}\mqty{1&1&1\\1&1&1\\1&1&1}.\label{f-box}
      \end{equation}
    \item Gaussian blur filter (GUS):\\
      \begin{equation}
        F_{\rm GUS}=\frac{1}{16}\mqty{1&2&1\\2&4&2\\1&2&1}.\label{f-gus}
      \end{equation}
    \item Edge detect A filter (EDA):\\
      \begin{equation}
        F_{\rm EDA}=\mqty{1&0&-1\\0&0&0\\-1&0&1}.\label{f-eda}
      \end{equation}
    \item Edge detect B filter (EDB):\\
      \begin{equation}
        F_{\rm EDB}=\mqty{0&-1&0\\-1&4&-1\\0&-1&0}.\label{f-edb}
      \end{equation}
    \item Edge detect C filter (EDC):\\
      \begin{equation}
        F_{\rm EDC}=\mqty{-1&-1&-1\\-1&8&-1\\-1&-1&-1}.\label{f-edc}
      \end{equation}
    \item Sharpen filter (SHP):\\
      \begin{equation}
        F_{\rm SHP}=\mqty{0&-1&0\\-1&5&-1\\0&-1&0}.\label{f-shp}
      \end{equation}
    \item Emboss filter (EMB):\\
      \begin{equation}
        F_{\rm EMB}=\mqty{-2&-1&0\\-1&1&1\\0&1&2}.\label{f-emb}
      \end{equation}
  \end{enumerate}
  Before discussing the existence of unique solutions of the convolution equations~\eqref{eq:convolution} for each filter, let us recall the following facts:
  \begin{lmm}\label{lmm:trieig}
    Let $A\in\mathbb{R}^{n\times n}$ be a tridiagonal Toeplitz matrix:
    \begin{equation}
      A = \rmtridiag(\beta, \alpha, \gamma) = 
      \mqty{
        \alpha & \beta  &        &        \\
        \gamma & \alpha & \ddots &        \\
               & \ddots & \ddots & \beta  \\
               &        & \gamma & \alpha
      }.
    \end{equation}
    Then, the eigenvalues of $A$ are given by
    \begin{equation}
      \lambda_k=\alpha + 2\sqrt{\beta\gamma}\cos\left(\frac{k\pi}{n+1}\right), \qquad k=1,2,\ldots,n.
    \end{equation}
  \end{lmm}
  \begin{proof}
    See \cite{smith1985}.
  \end{proof}
  \begin{lmm}\label{lmm:circeig}
    Let $A\in\mathbb{R}^{n\times n}$ be a circulant matrix:
    \begin{equation}
      A = \mqty{
        a_0     & a_1     & a_2    & \cdots & a_{n-1} \\
        a_{n-1} & a_0     & a_1    & \cdots & a_{n-2} \\
        a_{n-2} & a_{n-1} & a_0    & \cdots & a_{n-3} \\
        \vdots  & \vdots  & \vdots & \cdots & \vdots \\
        a_1     & a_2     & a_3    & \cdots & a_0
      }.
    \end{equation}
    Then, the eigenvalues of $A$ are given by
    \begin{equation}
      \lambda_k=\sum^{n-1}_{l=0}a_l\exp\left(-\frac{2kl\pi\im}{n}\right), \qquad k=1,2,\ldots,n.
    \end{equation}
  \end{lmm}
  \begin{proof}
    See \cite{gray2006}.
  \end{proof}
  \begin{lmm}\label{lmm:qteig}
    Let $A\in\mathbb{R}^{n\times n}$ be the following tridiagonal matrix:
    \begin{equation}
      A = 
      \mqty{
        \alpha+\gamma & \beta  &        &        &  \\
        \gamma        & \alpha & \beta  &        &  \\
                      & \gamma & \ddots & \ddots &  \\
                      &        & \ddots & \alpha & \beta \\
                      &        &        & \gamma & \alpha + \beta
      }.
    \end{equation}
    Then, the eigenvalues of $A$ are given by
    \begin{align}
      \begin{aligned}
        \lambda_k=\begin{cases}
          \alpha + 2\sqrt{\beta\gamma}\cos\left(\frac{k\pi}{n}\right), & k=1,2,\ldots,n-1,\\
          \alpha + \beta + \gamma, & k=n.
      \end{cases}
    \end{aligned}
  \end{align}
  \end{lmm}
  \begin{proof}
    The proof of this Lemma is given in Appendix.
  \end{proof}
    Note that eigenvalues of several tridiagonal matrices, including the matrix in Lemma~\ref{lmm:qteig}, are presented in~\cite{losonczi1992}. 
    However, in Appendix, we prove Lemma~\ref{lmm:qteig} differently from~\cite{losonczi1992}.
    For eigenvalues of more general tridiagonal matrices and their applications, we refer the reader to~\cite{dafonseca2020,losonczi1992}.

  \subsection{Box blur\label{subsec3.1}}
  \subsubsection{Zero boundary condition}
    When $F$ is the box blur filter, Eq.~\eqref{eq:gen-Sylvester} becomes
    \begin{equation}
      T_mXT_n=9B, \label{eq:box}
    \end{equation}
    where $T_n:=L_n + I_n + U_n
    \in\mathbb{R}^{n\times n}$.
    Therefore, the following result holds:
    \begin{cor}\label{cor:box}
      {\bfseries(Zero boundary condition)}
      Let $F\in\mathbb{R}^{3\times 3}$ be the box blur filter~\eqref{f-box}.
      Then, Eq.~\eqref{eq:convolution} with the zero boundary condition has a unique solution for every $B$ if and only if 
      $m,n\notin\{3l-1: l\in\mathbb{N}\}$.
    \end{cor}
    \begin{proof}
      From Theorem~\ref{thm:zero} and the box blur filter~\eqref{f-box}, Eq.~\eqref{eq:convolution} with the zero boundary condition is equivalent to~\eqref{eq:box}. 
      Therefore, we only need to show necessary and sufficient conditions for the existence of a unique solution of~\eqref{eq:box}.
      Eq.~\eqref{eq:box} has a unique solution for every $B$ if and only if $T_m$ and $T_n$ are nonsingular.
      From Lemma~\ref{lmm:trieig}, the eigenvalues of $T_n$ are given by
      \begin{equation}
        \lambda_k = 1 + 2\cos\left(\frac{k\pi}{n+1}\right), \qquad k=1,2,\ldots, n.
      \end{equation}
      Thus, $T_m$ and $T_n$ are nonsingular if and only if $m,n\notin\{3l-1: l\in\mathbb{N}\}$, which completes the proof.
    \end{proof}
  \subsubsection{Periodic boundary condition}
    Similarly, when $F$ is the box blur filter, Eq.~\eqref{eq:gen-Sylvester:p} becomes
    \begin{equation}
      T^{(\mathrm{P})}_m X T^{(\mathrm{P})}_n = 9B, \label{eq:boxp}
    \end{equation}
    where $T^{(\mathrm{P})}_n=L^{(\mathrm{P})}_n + I_n + U^{(\mathrm{P})}_n\in\mathbb{R}^{n\times n}$.
    Therefore, we obtain the following result:
    \begin{cor}\label{cor:pbox}
      {\bfseries(Periodic boundary condition)}
      Let $F\in\mathbb{R}^{3\times 3}$ be the box blur filter~\eqref{f-box}.
      Then, Eq.~\eqref{eq:convolution} with the periodic boundary condition has a unique solution for every $B$ if and only if 
      $m,n\notin\{3l: l\in\mathbb{N}\}$.
    \end{cor}
    \begin{proof}
      By an argument similar to Corollary~\ref{cor:box}, it suffices to show necessary and sufficient conditions for $T^{(\mathrm{P})}_m$ and $T^{(\mathrm{P})}_n$ in~\eqref{eq:boxp} to be nonsingular.
      From Lemma~\ref{lmm:circeig}, the eigenvalues of $T_n^{(\mathrm{P})}$ are given by
      \begin{equation}
        \lambda_k = 1 + \exp\left(-\frac{2k\pi\im}{n}\right) + \exp\left(\frac{2k\pi\im}{n}\right)
        = 1 + 2\cos\left(\frac{2k\pi}{n}\right), \qquad k=1,2,\ldots, n.
      \end{equation}
      Thus, $T_m^{(\mathrm{P})}$ and $T_n^{(\mathrm{P})}$ are nonsingular if and only if $m, n\notin\{3l: l\in\mathbb{N}\}$.
    \end{proof}
    \subsubsection{Reflexive boundary condition}
    When $F$ is the box blur filter,~\eqref{eq:gen-Sylvester:r} becomes
    \begin{equation}
      T^{(\mathrm{R})}_m X T^{(\mathrm{R})}_n = 9B, \label{eq:boxr}
    \end{equation}
    where $T^{(\mathrm{R})}_n=L^{(\mathrm{R})}_n + I_n + U^{(\mathrm{R})}_n\in\mathbb{R}^{n\times n}$.
    Therefore, we get the following result:
    \begin{cor}\label{cor:rbox}
      {\bfseries(Reflexive boundary condition)}
      Let $F\in\mathbb{R}^{3\times 3}$ be the box blur filter~\eqref{f-box}.
      Then, Eq.~\eqref{eq:convolution} with the reflexive boundary condition has a unique solution for every $B$ if and only if $m, n\notin \{3l:l\in\mathbb{N}\}$. 
    \end{cor}
    \begin{proof}
      By a similar argument to the above corollaries, we only need to show necessary and sufficient conditions for $T^{(\mathrm{R})}_m$ and $T^{(\mathrm{R})}_n$ in~\eqref{eq:boxr} to be nonsingular.
      From Lemma~\ref{lmm:qteig}, the eigenvalues of $T_n^{(\mathrm{R})}$ are given by
      \begin{equation}
        \lambda_k = 
        \begin{cases}
          1 + 2\cos\left(\frac{k\pi}{n}\right),&\qquad k=1,2,\ldots,n-1,\\
          3, &\qquad k=n.
        \end{cases}
      \end{equation}
      Thus, $T_m^{(\mathrm{R})}$ and $T_n^{(\mathrm{R})}$ are nonsingular if and only if $m, n\notin\{3l: l\in\mathbb{N}\}$.
    \end{proof}

  \subsection{Gaussian blur}
  \subsubsection{Zero boundary condition}
    When $F$ is the Gaussian blur filter, Eq.~\eqref{eq:gen-Sylvester} becomes
    \begin{equation}
      V_mXV_n=16B, \label{eq:gus}
    \end{equation}
    where $V_n:=L_n + 2I_n + U_n
    \in\mathbb{R}^{n\times n}$.
    Therefore we have the following result:
    \begin{cor}\label{cor:gus}
      {\bfseries(Zero boundary condition)}
      Let $F\in\mathbb{R}^{3\times 3}$ be the Gaussian blur filter~\eqref{f-gus}.
      Then, Eq.~\eqref{eq:convolution} with the zero boundary condition has a unique solution for every $B$.
    \end{cor}
    \begin{proof}
      Since~\eqref{eq:gus} has the same form as one in the case of the box blur filter, to complete the proof, we need only prove that $V_m$ and $V_n$ in~\eqref{eq:gus} are always nonsingular.
      From Lemma~\ref{lmm:trieig}, the eigenvalues of $V_n$ are given by
      \begin{equation}
        \lambda_k = 2 + 2\cos\left(\frac{k\pi}{n+1}\right), \qquad k=1,2,\ldots,n.
      \end{equation}
      Thus, it follows that $-1<\cos\left(\frac{k\pi}{n+1}\right)<1$ for $k\in\{1,2,\ldots, n\}$, which implies $0<\lambda_k<4$.
      Hence, $V_m$ and $V_n$ are always nonsingular.
    \end{proof}
  \subsubsection{Periodic boundary condition}
    When $F$ is the Gaussian filter, Eq.~\eqref{eq:gen-Sylvester:p} becomes
    \begin{equation}
      V_m^{\mathrm{(P)}} XV_n^{\mathrm{(P)}} = 16B, \label{eq:gusp}
    \end{equation}
    where $V_n^{\mathrm{(P)}}:=L_n^{\mathrm{(P)}} + 2I_n + U_n^{\mathrm{(P)}}\in\mathbb{R}^{n\times n}$.
    Therefore, the following result holds:
    \begin{cor}
      {\bfseries(Periodic boundary condition)}
      Let $F\in\mathbb{R}^{3\times 3}$ be the Gaussian blur filter~\eqref{f-gus}.
      Then, Eq.~\eqref{eq:convolution} with the periodic boundary condition has a unique solution for every $B$ if and only if $m,n\notin\{2l:l\in\mathbb{N}\}$. 
    \end{cor}
    \begin{proof}
      By a similar argument to the corollaries in Subsection~\ref{subsec3.1}, it is sufficient to provide necessary and sufficient conditions for $V_m^{\mathrm{(P)}}$ and $V_n^{\mathrm{(P)}}$ in~\eqref{eq:gusp} to be nonsingular.
      From Lemma~\ref{lmm:circeig}, the eigenvalues of $V_n^{(\mathrm{P})}$ are given by
      \begin{equation}
        \lambda_k = 2 + \exp\left(-\frac{2kl\pi\im}{n}\right) + \exp\left(\frac{2kl\pi\im}{n}\right)= 2 + 2\cos\left(\frac{2k\pi}{n}\right), \qquad k=1,2,\ldots,n.
      \end{equation}
      Thus, 
      $V_m^{(\mathrm{P})}$ and $V_n^{(\mathrm{P})}$ are nonsingular if and only if $m,n\notin\{2l:l\in\mathbb{N}\}$.
      Hence we complete the proof.
    \end{proof}
  \subsubsection{Reflexive boundary condition}
    When $F$ is the Gaussian filter, Eq.~\eqref{eq:gen-Sylvester:r} becomes
    \begin{equation}
      V_m^{\mathrm{(R)}} XV_n^{\mathrm{(R)}} = 16B, \label{eq:gusr}
    \end{equation}
    where $V_n^{\mathrm{(R)}}:=L_n^{\mathrm{(R)}} + 2I_n + U_n^{\mathrm{(R)}}\in\mathbb{R}^{n\times n}$.
    Therefore, the following result holds:
    \begin{cor}
      {\bfseries(Reflexive boundary condition)}
      Let $F\in\mathbb{R}^{3\times 3}$ be the Gaussian blur filter~\eqref{f-gus}.
      Then, Eq.~\eqref{eq:convolution} with the reflexive boundary condition has a unique solution for every $B$.
    \end{cor}
    \begin{proof}
      It suffices to demonstrate that $V_m^{\mathrm{(R)}}$ and $V_n^{\mathrm{(R)}}$ in~\eqref{eq:gusr} are nonsingular, as in Corollary~\ref{cor:gus}.
      From Lemma~\ref{lmm:qteig}, the eigenvalues of $V_n^{(\mathrm{R})}$ are given by
      \begin{equation}
        \lambda_k = 
        \begin{cases}
          2 + 2\cos\left(\frac{k\pi}{n}\right), &\qquad k=1,2,\ldots,n-1,\\
          4, &\qquad k=n.
        \end{cases}
      \end{equation}
      Thus, it follows that $-1<\cos\left(\frac{k\pi}{n}\right)<1$ for $k\in\{1,2,\ldots, n-1\}$, which implies $\lambda_k>0$.
      Hence, $V_m^{(\mathrm{P})}$ and $V_n^{(\mathrm{P})}$ are always nonsingular.
    \end{proof}

    \subsection{Edge detect A}
    \subsubsection{Zero boundary condition}
    When $F$ is the edge detect A filter, Eq.~\eqref{eq:gen-Sylvester} becomes
    \begin{equation}
      S_mXS_n^\top=B, \label{eq:eda}
    \end{equation}
    where $S_n:=U_n - L_n\in\mathbb{R}^{n\times n}$.
    Therefore, the following result holds:
    \begin{cor}
      {\bfseries(Zero boundary condition)}
      Let $F\in\mathbb{R}^{3\times 3}$ be the edge detect A filter~\eqref{f-eda}.
      Then, Eq.~\eqref{eq:convolution} with the zero boundary condition has a unique solution for every $B$ if and only if $m,n\notin\{2l-1:l\in\mathbb{N}\}$. 
    \end{cor}
    \begin{proof}
      The proof is completed by showing that $S_m$ and $S_n$ in~\eqref{eq:eda} are nonsingular if and only if $m,n\notin\{2l-1:l\in\mathbb{N}\}$.
      From Lemma~\ref{lmm:trieig}, the eigenvalues of $S_n$ are given by
      \begin{equation}
        \lambda_k = 2\im\cos\left(\frac{k\pi}{n+1}\right), \qquad k=1,2,\ldots,n.
      \end{equation}
      Thus, $S_m$ and $S_n$ are nonsingular if and only if $m,n\notin\{2l-1:l\in\mathbb{N}\}$.
    \end{proof}
    \subsubsection{Periodic boundary condition}
    When $F$ is the edge detect A filter, Eq.~\eqref{eq:gen-Sylvester:p} becomes
    \begin{equation}
      S_m^{(\mathrm{P})}X{S_n^{(\mathrm{P})}}^\top=B, \label{eq:edap}
    \end{equation}
    where $S_n^{(\mathrm{P})}:=U_n^{(\mathrm{P})} - L_n^{(\mathrm{P})}\in\mathbb{R}^{n\times n}$.
    Therefore, the following result holds:
    \begin{cor}
      {\bfseries(Periodic boundary condition)}
      Let $F\in\mathbb{R}^{3\times 3}$ be the edge detect A filter~\eqref{f-eda}.
      Then, Eq.~\eqref{eq:convolution} with the periodic boundary condition does not have a unique solution for every $B$.
    \end{cor}
    \begin{proof}
      To complete the proof, it is sufficient to prove that $S_m^{(\mathrm{P})}$ or $S_m^{(\mathrm{P})}$ in~\eqref{eq:edap} is singular.
      From Lemma~\ref{lmm:circeig}, the eigenvalues of $S_n^{(\mathrm{P})}$ are given by
      \begin{equation}
        \lambda_k = \exp\left(-\frac{2k\pi\im}{n}\right) - \exp\left(\frac{2k\pi\im}{n}\right)
        = -2\im\sin\left(\frac{2k\pi}{n}\right), \qquad k=1,2,\ldots,n,
      \end{equation}
      which implies $\lambda_n=0$ for any $n\in\mathbb{N}$.
      Hence, $S_m^{(\mathrm{P})}$ and $S_n^{(\mathrm{P})}$ are always singular.
    \end{proof}
  \subsubsection{Reflexive boundary condition}
    When $F$ is the edge detect A filter, Eq.~\eqref{eq:gen-Sylvester:r} becomes
    \begin{equation}
      S_m^{(\mathrm{R})}X{S_n^{(\mathrm{R})}}^\top=B, \label{eq:edar}
    \end{equation}
    where $S_n^{(\mathrm{R})}:=U_n^{(\mathrm{R})} - L_n^{(\mathrm{R})}\in\mathbb{R}^{n\times n}$.
    Therefore, the following result holds:
    \begin{cor}
      {\bfseries(Reflexive boundary condition)}
      Let $F\in\mathbb{R}^{3\times 3}$ be the edge detect A filter~\eqref{f-eda}.
      Then, Eq.~\eqref{eq:convolution} with the reflexive boundary condition does not have a unique solution for every $B$.
    \end{cor}
    \begin{proof}
      To complete the proof, it is sufficient to prove that $S_m^{(\mathrm{R})}$ or $S_m^{(\mathrm{R})}$ in~\eqref{eq:edar} is singular.
      From Lemma~\ref{lmm:qteig}, 
      zero is an eigenvalue of $S_n^{(\mathrm{R})}$ for any $n\in\mathbb{N}$.
      Hence, $S_m^{(\mathrm{R})}$ and $S_n^{(\mathrm{R})}$ are always singular.
    \end{proof}

    \subsection{Edge detect B}
    \subsubsection{Zero boundary condition}
    When $F$ is the edge detect B filter, Eq.~\eqref{eq:gen-Sylvester} becomes the following Lyapunov equation
    \begin{equation}
      C_mX + XC_n = B, \label{eq:edb}
    \end{equation}
    where $C_n:= 2I_n - U_n - L_n
    \in\mathbb{R}^{n\times n}$.
    Therefore, we have the following result:
    \begin{cor}\label{cor:edb}
      {\bfseries(Zero boundary condition)}
      Let $F\in\mathbb{R}^{3\times 3}$ be the edge detect B filter~\eqref{f-edb}.
      Then, Eq.~\eqref{eq:convolution} with the zero boundary condition has a unique solution for every $B$.
    \end{cor}
    \begin{proof}
      To complete the proof, it suffices to prove Eq.~\eqref{eq:edb} has a unique solution for every $B$.
      The Lyapunov equation~\eqref{eq:edb} has a unique solution if and only if $C_m$ and $-C_n$ have no common eigenvalue.
      From Lemma~\ref{lmm:trieig}, the eigenvalues of $C_n$ are given by
      \begin{equation}
        \lambda_k = 2 + 2\cos\left(\frac{k\pi}{n+1}\right), \qquad k=1,2,\ldots,n,
      \end{equation}
      which implies that all eigenvalues of $C_m$ and $C_n$ are positive.
      Hence,~$C_m$ and $-C_n$ have no common eigenvalue, which completes the proof.
    \end{proof}
  \subsubsection{Periodic boundary condition}
    When $F$ is the edge detect B filter, Eq.~\eqref{eq:gen-Sylvester:p} becomes
    \begin{equation}
      C_m^{(\mathrm{P})}X + XC_n^{(\mathrm{P})}=B, \label{eq:edbp}
    \end{equation}
    where $C_n^{(\mathrm{P})}:=2I_n - U_n^{(\mathrm{P})} - L_n^{(\mathrm{P})}\in\mathbb{R}^{n\times n}$.
    Therefore, the following result holds:
    \begin{cor}\label{cor:edbp}
      {\bfseries(Periodic boundary condition)}
      Let $F\in\mathbb{R}^{3\times 3}$ be the edge detect B filter~\eqref{f-edb}.
      Then, Eq.~\eqref{eq:convolution} with the periodic boundary condition does not have a unique solution for every $B$.
    \end{cor}
    \begin{proof}
      It is sufficient to show that Eq.~\eqref{eq:edbp} does not have a unique solution for every $B$, i.e., $C_m^{(\mathrm{P})}$ and $-C_n^{(\mathrm{P})}$ always have common eigenvalues.
      From Lemma~\ref{lmm:circeig}, the eigenvalues of $C_n^{(\mathrm{P})}$ are given by
      \begin{equation}
        \lambda_k = 2 - \exp\left(-\frac{2k\pi\im}{n}\right) - \exp\left(\frac{2k\pi\im}{n}\right)
        = 2 - 2\cos\left(\frac{2k\pi}{n}\right), \qquad k=1,2,\ldots,n,
      \end{equation}
      which implies $\lambda_n=0$ for any $n\in\mathbb{N}$.
      Hence, $C_m^{(\mathrm{P})}$ and $-C_n^{(\mathrm{P})}$ always have the common eigenvalue $0$. 
    \end{proof}
  \subsubsection{Reflexive boundary condition}
    When $F$ is the edge detect B filter, Eq.~\eqref{eq:gen-Sylvester:r} becomes
    \begin{equation}
      C_m^{(\mathrm{R})}X + XC_n^{(\mathrm{R})}=B, \label{eq:edbr}
    \end{equation}
    where $C_n^{(\mathrm{R})}:=2I_n - U_n^{(\mathrm{R})} - L_n^{(\mathrm{R})}\in\mathbb{R}^{n\times n}$.
    Therefore, the following result holds:
    \begin{cor}
      {\bfseries(Reflexive boundary condition)}
      Let $F\in\mathbb{R}^{3\times 3}$ be the edge detect B filter~\eqref{f-edb}.
      Then, Eq.~\eqref{eq:convolution} with the reflexive boundary condition does not have a unique solution for every $B$.
    \end{cor}
    \begin{proof}
      By a similar argument to Corollary~\ref{cor:edbp}, we only need to show that $C_m^{(\mathrm{R})}$ and $C_n^{(\mathrm{R})}$ in~\eqref{eq:edbr} always have at least one eigenvalue in common.
      From Lemma~\ref{lmm:qteig}, the eigenvalues of $C_n^{(\mathrm{R})}$ are given by
      \begin{equation}
        \lambda_k = 
        \begin{cases}
          2 + 2\cos\left(\frac{k\pi}{n}\right), &\qquad k=1,2,\ldots,n-1,\\
          0, &\qquad k=n.
        \end{cases}
      \end{equation}
      Hence, $C_m^{(\mathrm{R})}$ and $-C_n^{(\mathrm{R})}$ always have the common eigenvalue $0$. 
    \end{proof}

    \subsection{Edge detect C}
    \subsubsection{Zero boundary condition}
    When $F$ is the edge detect C filter, Eq.~\eqref{eq:gen-Sylvester} becomes
    \begin{equation}
      9X - T_mXT_n = B, \label{eq:edc}
    \end{equation}
    where $T_n:=L_n + I_n + U_n\in\mathbb{R}^{n\times n}$.
    Therefore, we have the following result:
    \begin{cor}
      {\bfseries(Zero boundary condition)}
      Let $F\in\mathbb{R}^{3\times 3}$ be the edge detect C filter~\eqref{f-edc}.
      Then, Eq.~\eqref{eq:convolution} with the zero boundary condition has a unique solution for every $B$.
    \end{cor}
    \begin{proof}
      The proof is completed by showing that Eq.~\eqref{eq:edc} has a unique solution for every $B$.
      Let $\lambda_1^{(m)}, \ldots, \lambda_n^{(m)}$ and $\lambda_1^{(n)}, \ldots, \lambda_n^{(n)}$ be the eigenvalues of $T_m$ and $T_n$, respectively.
      Then, Eq.~\eqref{eq:edc} has a unique solution for every $B$ if and only if $\lambda_i^{(m)}\lambda_j^{(n)}\ne 9$ for $i\in\{1,2,\ldots,m\}$ and $j\in\{1,2,\ldots,n\}$.
      Therefore, we only need to show that $\lambda_i^{(m)}\lambda_j^{(n)}\ne 9$ for all $i,j$.
      From Lemma~\ref{lmm:trieig}, the eigenvalues of $T_n$ are given by
      \begin{equation}
        \lambda_k^{(n)} = 1 + 2\cos\left(\frac{k\pi}{n+1}\right), \qquad k=1,2,\ldots,n.
      \end{equation}
      Thus, it follows that $-1<\cos\left(\frac{k\pi}{n+1}\right)<1$ for $k\in\{1,2,\ldots, n\}$, which implies $-1<\lambda_k^{(n)}<3$.
      Hence, $\lambda_i^{(m)}\lambda_j^{(n)}<9$ for all $i,j$, which completes the proof.
    \end{proof}
  \subsubsection{Periodic boundary condition}
    When $F$ is the edge detect C filter, Eq.~\eqref{eq:gen-Sylvester:p} becomes
    \begin{equation}
      9X - T_m^{(\mathrm{P})}XT_n^{(\mathrm{P})}=B, \label{eq:edcp}
    \end{equation}
    where $T_n^{(\mathrm{P})}:=L_n^{(\mathrm{P})} + I_n + U_n^{(\mathrm{P})}\in\mathbb{R}^{n\times n}$.
    Therefore, the following result holds:
    \begin{cor}\label{cor:edcp}
      {\bfseries(Periodic boundary condition)}
      Let $F\in\mathbb{R}^{3\times 3}$ be the edge detect C filter~\eqref{f-edc}.
      Then, Eq.~\eqref{eq:convolution} with the periodic boundary condition does not have a unique solution for every $B$.
    \end{cor}
    \begin{proof}
      The proof is completed by proving that Eq.~\eqref{eq:edcp} does not have a unique solution for every $B$.
      Let c
      Eq.~\eqref{eq:edcp} does not have a unique solution for every $B$ if and only if there exist $\lambda_i^{(m)}$ and $\lambda_j^{(n)}$ such that $\lambda_i^{(m)}\lambda_j^{(n)}=9$.
      From Lemma~\ref{lmm:circeig}, the eigenvalues of $T_n^{(\mathrm{P})}$ are given by
      \begin{equation}
        \lambda_k^{(n)} = 1 + 2\cos\left(\frac{2k\pi}{n}\right), \qquad k=1,2,\ldots,n.
      \end{equation}
      Thus, it follows that $\lambda_m^{(m)}\lambda_n^{(n)}=9$, which completes the proof.
    \end{proof}
  \subsubsection{Reflexive boundary condition}
    When $F$ is the edge detect C filter, Eq.~\eqref{eq:gen-Sylvester:r} becomes
    \begin{equation}
      9X - T_m^{(\mathrm{R})}XT_n^{(\mathrm{R})}=B, \label{eq:edcr}
    \end{equation}
    where $T_n^{(\mathrm{R})}:=L_n^{(\mathrm{R})} + I_n + U_n^{(\mathrm{R})}\in\mathbb{R}^{n\times n}$.
    Therefore, the following result holds:
    \begin{cor}
      {\bfseries(Reflexive boundary condition)}
      Let $F\in\mathbb{R}^{3\times 3}$ be the edge detect C filter~\eqref{f-edc}.
      Then, Eq.~\eqref{eq:convolution} with the reflexive boundary condition does not have a unique solution for every $B$.
    \end{cor}
    \begin{proof}
      Let $\lambda_1^{(m)}, \ldots, \lambda_m^{(m)}$ and $\lambda_1^{(n)}, \ldots, \lambda_n^{(n)}$ be the eigenvalues of $T_m^{(\mathrm{R})}$ and $T_n^{(\mathrm{R})}$ in~\eqref{eq:edcr}, respectively.
      By a similar argument to Corollary~\ref{cor:edcp}, we only need to show that there exist $\lambda_i^{(m)}$ and $\lambda_j^{(n)}$ such that $\lambda_i^{(m)}\lambda_j^{(n)}=9$.
      From Lemma~\ref{lmm:qteig}, the eigenvalues of $T_n^{(\mathrm{R})}$ are given by
      \begin{equation}
        \lambda_k^{(n)} = 
        \begin{cases}
          1 + 2\cos\left(\frac{k\pi}{n}\right),&\qquad k=1,2,\ldots,n-1,\\
          3, &\qquad k=n,
        \end{cases}
      \end{equation}
      which implies $\lambda_m^{(m)}\lambda_n^{(n)}=9$.
    \end{proof}

    \subsection{Sharpen}
    \subsubsection{Zero boundary condition}
    When $F$ is the sharpen filter, Eq.~\eqref{eq:gen-Sylvester} becomes
    \begin{equation}
      \tilde{C}_mX + X\tilde{C}_n=B, \label{eq:shp}
    \end{equation}
    where $\tilde{C}_n:=\frac{5}{2}I_n - U_n - L_n
    \in\mathbb{R}^{n\times n}$.
    From this, the following result holds:
    \begin{cor}
      {\bfseries(Zero boundary condition)}
      Let $F\in\mathbb{R}^{3\times 3}$ be the sharpen filter~\eqref{f-shp}.
      Then, Eq.~\eqref{eq:convolution} with the zero boundary condition has a unique solution for every $B$.
    \end{cor}
    \begin{proof}
      By a similar argument to Corollary~\ref{cor:edb}, it suffices to prove that $\tilde{C}_m$ and $-\tilde{C}_n$ in~\eqref{eq:shp} have no common eigenvalue.
      From Lemma~\ref{lmm:trieig}, the eigenvalues of $\tilde{C}_n$ are given by
      \begin{equation}
        \lambda_k = \frac{5}{2} + 2\cos\left(\frac{k\pi}{n+1}\right), \qquad k=1,2,\ldots,n.
      \end{equation}
      Thus, it follows that $-1<\cos\left(\frac{k\pi}{n+1}\right)<1$ for $k\in\{1,2,\ldots,n\}$, which implies $\frac{1}{2}<\lambda_k<\frac{9}{2}$.
      Hence, $\tilde{C}_m$ and $-\tilde{C}_n$ does not have common eigenvalues.
    \end{proof}
  \subsubsection{Periodic boundary condition}
    When $F$ is the sharpen filter, Eq.~\eqref{eq:gen-Sylvester:p} can be rewritten as
    \begin{equation}
      \tilde{C}_m^{(\mathrm{P})}X + X\tilde{C}_n^{(\mathrm{P})} = B,\label{eq:shpp}
    \end{equation}
    where $\tilde{C}_n^{(\mathrm{P})}:=\frac{5}{2}I_n - U_n^{(\mathrm{P})} - L_n^{(\mathrm{P})}\in\mathbb{R}^{n\times n}$.
    Hence, the following result holds:
    \begin{cor}
      {\bfseries(Periodic boundary condition)}
      Let $F\in\mathbb{R}^{3\times 3}$ be the sharpen filter~\eqref{f-shp}.
      Then, Eq.~\eqref{eq:convolution} with the periodic boundary condition has a unique solution for every $B$.
    \end{cor}
    \begin{proof}
      To complete the proof, we only need to prove that $\tilde{C}_m^{(\mathrm{P})}$ and $-\tilde{C}_n^{(\mathrm{P})}$ in~\eqref{eq:shpp} have no common eigenvalue.
      From Lemma~\ref{lmm:circeig}, the eigenvalues of $\tilde{C}_n^{(\mathrm{P})}$ are given by
      \begin{equation}
        \lambda_k = \frac{5}{2} - \exp\left(-\frac{2k\pi\im}{n}\right) - \exp\left(\frac{2k\pi\im}{n}\right) = \frac{5}{2} - 2\cos\left(\frac{2k\pi}{n}\right), \qquad k=1,2,\ldots,n.
      \end{equation}
      Thus, it follows that $-1\le\cos\left(\frac{2k\pi}{n}\right)\le1$ for any $k\in\{1,2,\ldots,n\}$, which implies $\frac{1}{2}\le\lambda_k\le\frac{9}{2}$.
      Hence, $\tilde{C}_m^{(\mathrm{P})}$ and $-\tilde{C}_n^{(\mathrm{P})}$ have no common eigenvalue.
    \end{proof}
  \subsubsection{Reflexive boundary condition}
    When $F$ is the sharpen filter, Eq.~\eqref{eq:gen-Sylvester:r} can be rewritten as
    \begin{equation}
      \tilde{C}_m^{(\mathrm{R})}X + X\tilde{C}_n^{(\mathrm{R})} = B,\label{eq:shpr}
    \end{equation}
    where $\tilde{C}_n^{(\mathrm{R})}:=\frac{5}{2}I_n - U_n^{(\mathrm{R})} - L_n^{(\mathrm{R})}\in\mathbb{R}^{n\times n}$.
    Hence, the following result holds:
    \begin{cor}
      {\bfseries(Reflexive boundary condition)}
      Let $F\in\mathbb{R}^{3\times 3}$ be the sharpen filter~\eqref{f-shp}.
      Then, Eq.~\eqref{eq:convolution} with the reflexive boundary condition has a unique solution for every $B$.
    \end{cor}
    \begin{proof}
      The proof is completed by showing that $\tilde{C}_m^{(\mathrm{R})}$ and $-\tilde{C}_n^{(\mathrm{R})}$ in~\eqref{eq:shpr} have no common eigenvalue.
      From Lemma~\ref{lmm:qteig}, the eigenvalues of $\tilde{C}_n^{(\mathrm{R})}$ are given by
      \begin{equation}
        \lambda_k =
        \begin{cases}
          \frac{5}{2} + 2\cos\left(\frac{k\pi}{n}\right),&\qquad k=1,2,\ldots,n-1,\\
          \frac{1}{2}, &\qquad k=n.
        \end{cases}
      \end{equation}
      Thus, it follows that $-1\le\cos\left(\frac{k\pi}{n}\right)\le1$ for any $k\in\{1,2,\ldots,n-1\}$, which implies $\frac{1}{2}\le\lambda_k\le\frac{9}{2}$.
      Hence, $\tilde{C}_m^{(\mathrm{R})}$ and $-\tilde{C}_n^{(\mathrm{R})}$ have no common eigenvalue.
    \end{proof}

  \subsection{Emboss}
    In this subsection, we just characterize the unique solvability of the convolution equation~\eqref{eq:convolution} in the case of the zero or periodic boundary condition. 
    \subsubsection{Zero boundary condition}
    When $F$ is the emboss filter, Eq.~\eqref{eq:gen-Sylvester} becomes
    \begin{equation}
      X+\frac{1}{2}\left(\tilde{L}_m X\tilde{L}^\top_n - \tilde{L}^\top_mX\tilde{L}_n\right)=B \label{eq:emb}
    \end{equation}
    where $\tilde{L}_n:=I_n + 2L_n\in\mathbb{R}^{n\times n}$.
    Therefore, we have the following results:
    \begin{cor}\label{cor:emb}
      {\bfseries(Zero boundary condition)}
      Let $F\in\mathbb{R}^{3\times 3}$ be the emboss filter~\eqref{f-emb}.
      Then, Eq.~\eqref{eq:convolution} with the zero boundary condition has a unique solution for every $B$.
    \end{cor}
    \begin{proof}
      The proof is completed by showing that Eq.~\eqref{eq:emb} always has a unique solution for every $B$.
      Applying the vec operator to~\eqref{eq:emb} yields
      \begin{equation}
        \left[I_n\otimes I_m + \frac{1}{2}\left(\tilde{L}_n\otimes \tilde{L}_m - \tilde{L}_n^\top \otimes \tilde{L}_m^\top\right)\right]\bm{x} = \bm{b}.\label{eq:emb-ls}
      \end{equation}
      Then, the coefficient matrix of~\eqref{eq:emb-ls} is nonsingular because it is a sum of the identity matrix $I_n\otimes I_m$ and the skew-symmetric matrix $\frac{1}{2}\left(\tilde{L}_n\otimes \tilde{L}_m - \tilde{L}_n^\top \otimes \tilde{L}_m^\top\right)$.
      Thus, ~\eqref{eq:emb} always has a unique soluton for every $B$, which completes the proof.
    \end{proof}
  \subsubsection{Periodic boundary condition}
    When $F$ is the emboss filter, Eq.~\eqref{eq:gen-Sylvester:p} becomes
    \begin{equation}
      X+\frac{1}{2}\left(\tilde{L}_m^{(\mathrm{P})}X\left({\tilde{L}_n}^{(\mathrm{P})}\right)^\top - \left(\tilde{L}_m^{(\mathrm{P})}\right)^\top X\tilde{L}_n^{(\mathrm{P})}\right)=B \label{eq:embp}
    \end{equation}
    where $\tilde{L}_n^{(\mathrm{P})}:=I_n + 2L_n^{(\mathrm{P})}\in\mathbb{R}^{n\times n}$.
    \begin{cor}
      {\bfseries(Periodic boundary condition)}
      Let $F\in\mathbb{R}^{3\times 3}$ be the emboss filter~\eqref{f-emb}.
      Then, Eq.~\eqref{eq:convolution} with the periodic boundary condition has a unique solution for every $B$.
    \end{cor}
    \begin{proof}
      By a similar argument to Corollary~\ref{cor:emb}, it is sufficient to show that the coefficient matrix of the linear system obtained by vectorizing~\eqref{eq:embp} is nonsingular.
      Applying the vec operator to~\eqref{eq:embp} yields
      \begin{equation}
        \left\{I_n\otimes I_m + \frac{1}{2}\left[\tilde{L}^{(\mathrm{P})}_n\otimes \tilde{L}_m^{(\mathrm{P})} - \left(\tilde{L}_n^{(\mathrm{P})}\right)^\top \otimes \left(\tilde{L}_m^{(\mathrm{P})}\right)^\top\right]\right\}\bm{x} = \bm{b}.\label{eq:embp-ls}
      \end{equation}
      Then, the coefficient matrix of~\eqref{eq:embp-ls} is nonsingular because it is a sum of the identity matrix $I_n\otimes I_m$ and the skew-symmetric matrix $\frac{1}{2}\left[\tilde{L}^{(\mathrm{P})}_n\otimes \tilde{L}_m^{(\mathrm{P})} - \left(\tilde{L}_n^{(\mathrm{P})}\right)^\top \otimes \left(\tilde{L}_m^{(\mathrm{P})}\right)^\top\right]$.
    \end{proof}
  \subsubsection{Reflexive boundary condition\label{subsubsec3.7.3}}
    When $F$ is the emboss filter, Eq.~\eqref{eq:gen-Sylvester:r} becomes
    \begin{equation}
      X+\frac{1}{2}\left[\tilde{L}_m^{(\mathrm{R})}X\left({\tilde{L}_n}^{(\mathrm{R})}\right)^\top - \tilde{U}_m^{(\mathrm{R})}X\left({\tilde{U}_n}^{(\mathrm{R})}\right)^\top\right]=B \label{eq:embr}
    \end{equation}
    where $\tilde{L}_n^{(\mathrm{R})} := I_n + 2L_n^{(\mathrm{R})}\in\mathbb{R}^{n\times n}$ and $\tilde{U}_n^{(\mathrm{R})} := I_n + 2U_n^{(\mathrm{R})}\in\mathbb{R}^{n\times n}$.
    Unfortunately, unlike the above cases, we cannot show the necessary and sufficient conditions for a unique solution at present.
    However, we can infer from the following discussion that Eq.~\eqref{eq:embr} has a unique solution for every $B$.
    
    Vectorizing Eq.~\eqref{eq:embr} yields the linear system~\eqref{eq:ls} with the coefficient matrix
    \begin{equation}
      \mathcal{F} = I_n\otimes I_m + \frac{1}{2}\left(\tilde{L}^{(\mathrm{R})}_n\otimes \tilde{L}_m^{(\mathrm{R})} - \tilde{U}_n^{(\mathrm{R})} \otimes \tilde{U}_m^{(\mathrm{R})}\right)\in\mathbb{R}^{mn\times mn}. \label{coef:embr}
    \end{equation}
    Therefore, the uniqueness of solutions of Eq.~\eqref{eq:embr} is reduced to the nonsingularity of~\eqref{coef:embr}.
    To support this, we provide some examples of the eigenvalue distribution of $\mathcal{F}$ in Figure~\ref{fig:eigdist} and check that zero is not an eigenvalue of~$\mathcal{F}$.
    For all cases of matrix sizes shown in Figure~\ref{fig:eigdist}, the real parts of all eigenvalues are $1$.
    Thus, from the examples, it can be conjectured that all eigenvalues of $\mathcal{F}$ for all $m,n$ can be written in the form $\lambda = 1 + \im\alpha(\ne 0), \alpha\in\mathbb{R}$.
    \begin{figure}[htbp]
      \centering
      \begin{tabular}{ccc}
        \begin{minipage}[t]{0.3\linewidth}
          \centering
          \includegraphics[scale=0.3]{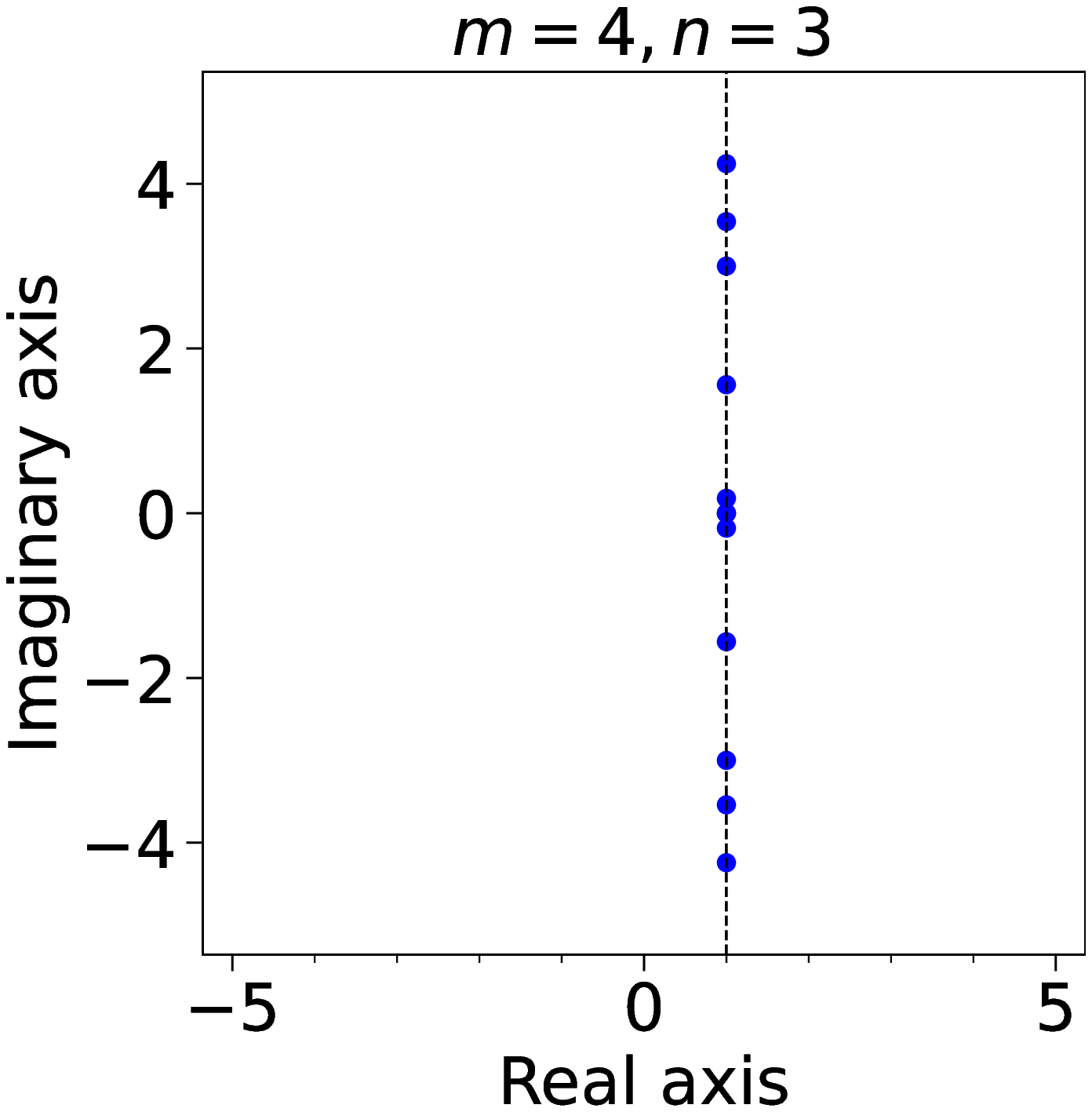}
          \label{fig:embr_4_3}
        \end{minipage} &
        \begin{minipage}[t]{0.3\linewidth}
          \centering
          \includegraphics[scale=0.3]{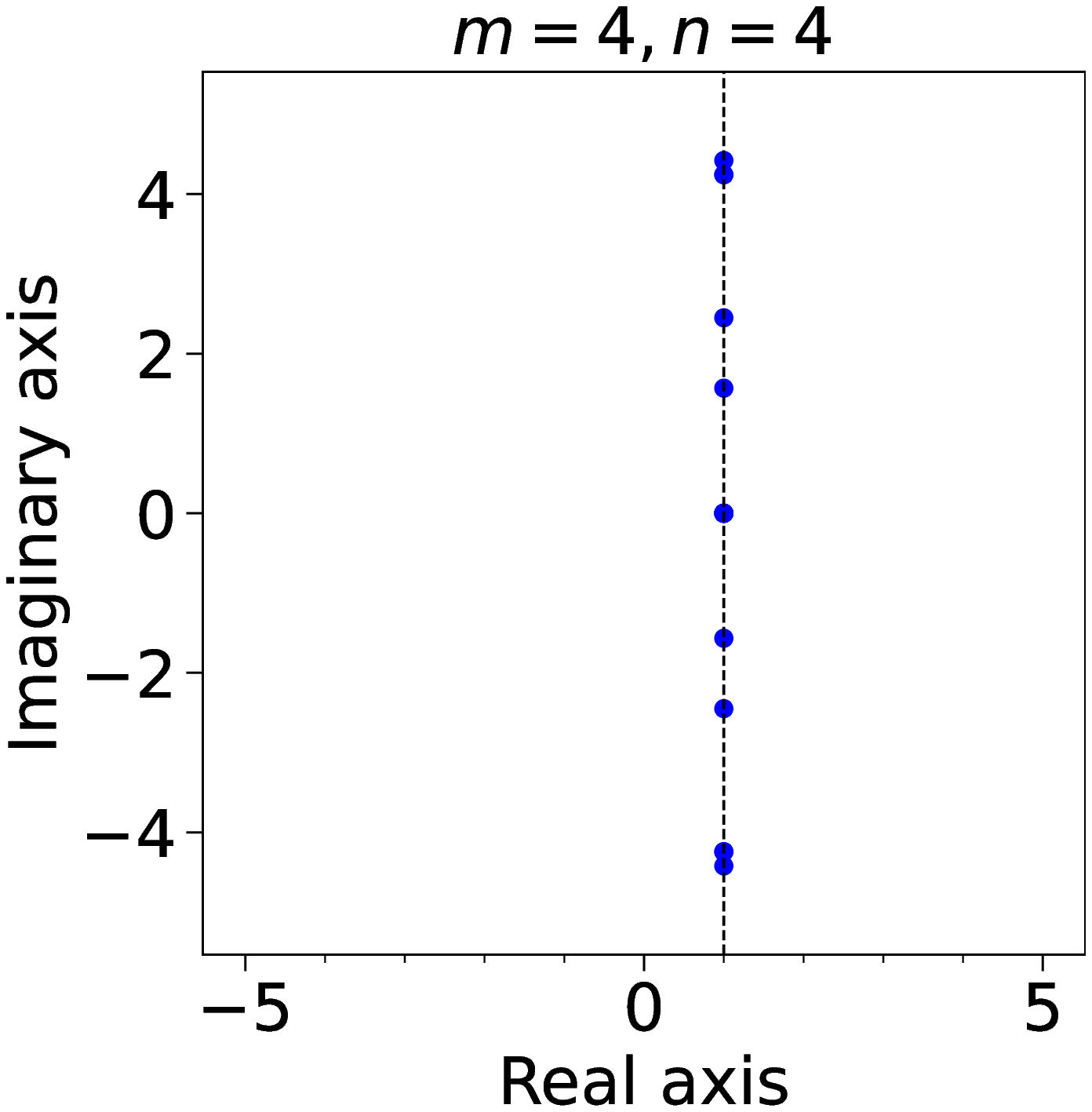}
          \label{fig:embr_4_4}
        \end{minipage} &
        \begin{minipage}[t]{0.3\linewidth}
          \centering
          \includegraphics[scale=0.3]{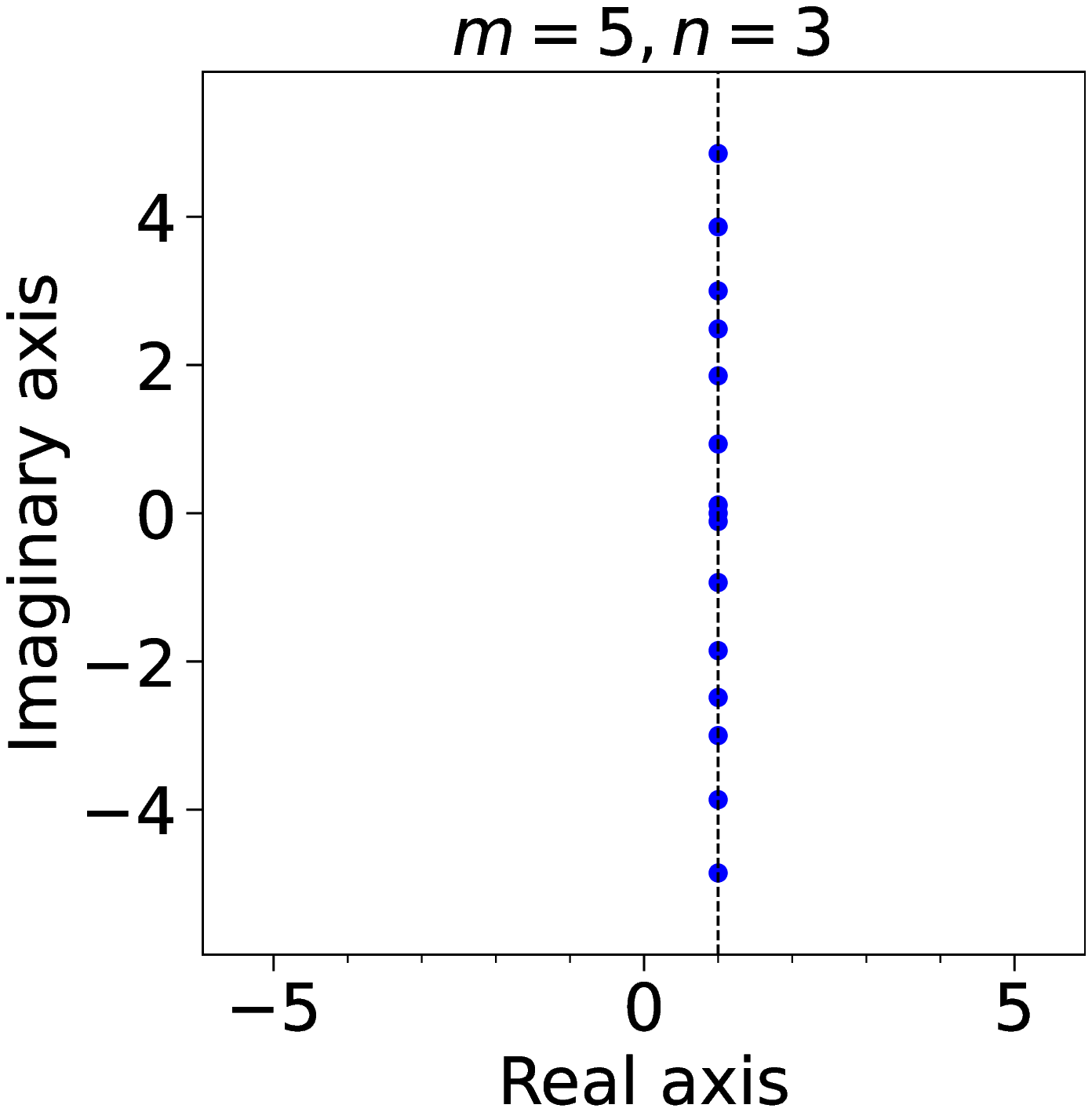}
          \label{fig:embr_5_3}
        \end{minipage} \\
        \begin{minipage}[t]{0.3\linewidth}
          \centering
          \includegraphics[scale=0.3]{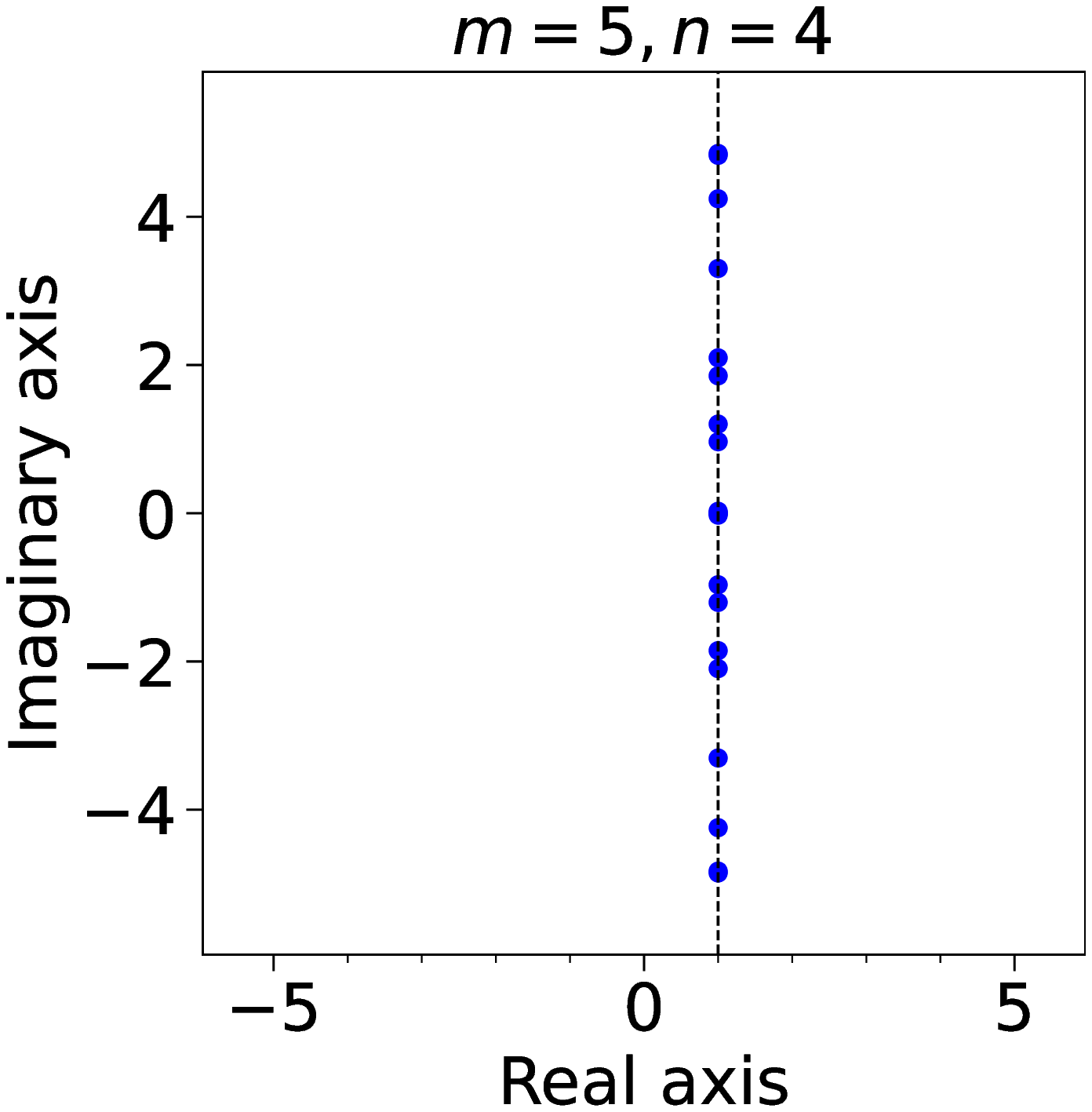}
          \label{fig:embr_5_4}
        \end{minipage} &
        \begin{minipage}[t]{0.3\linewidth}
          \centering
          \includegraphics[scale=0.3]{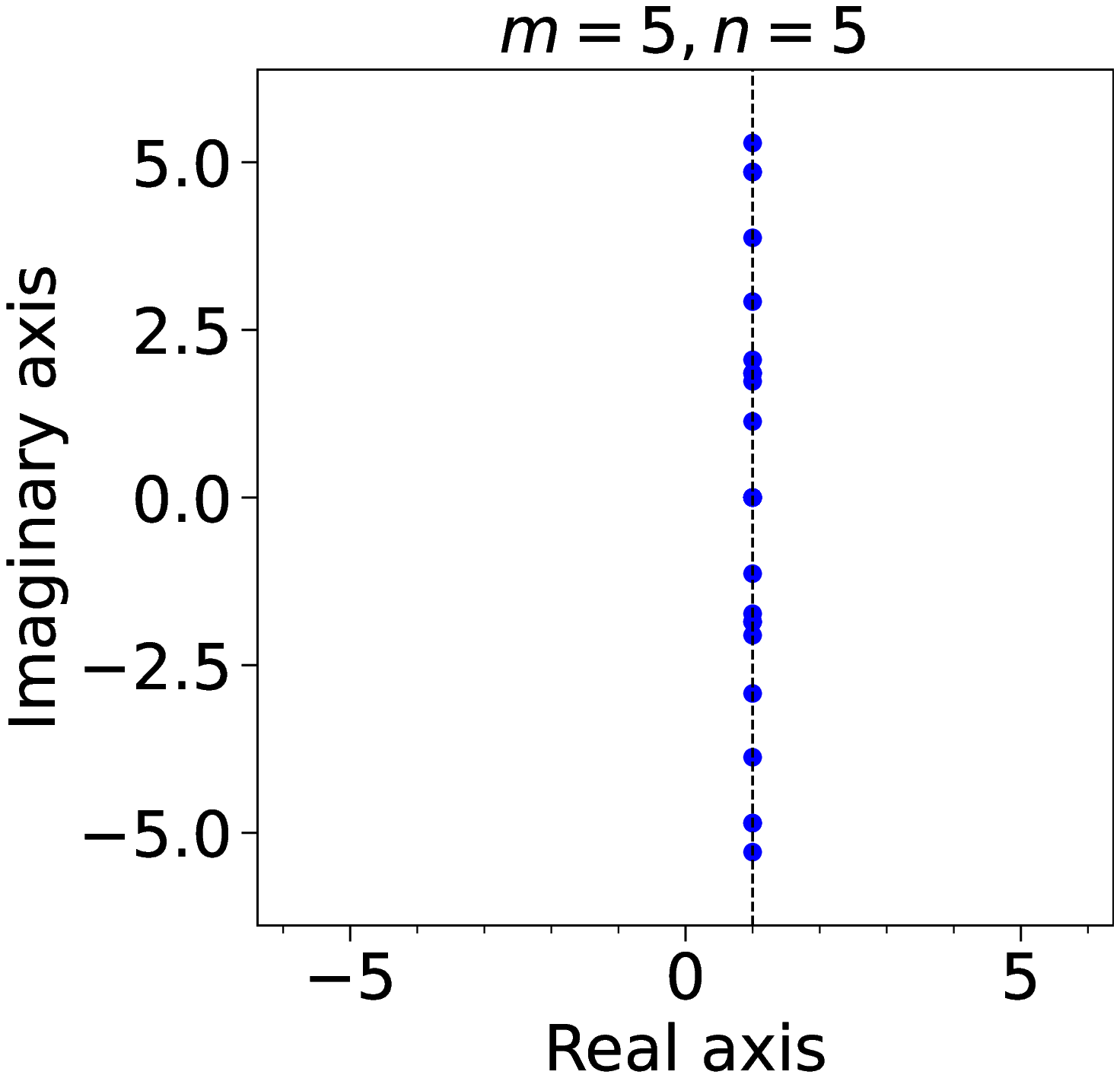}
          \label{fig:embr_5_5}
        \end{minipage} &
        \begin{minipage}[t]{0.3\linewidth}
          \centering
          \includegraphics[scale=0.3]{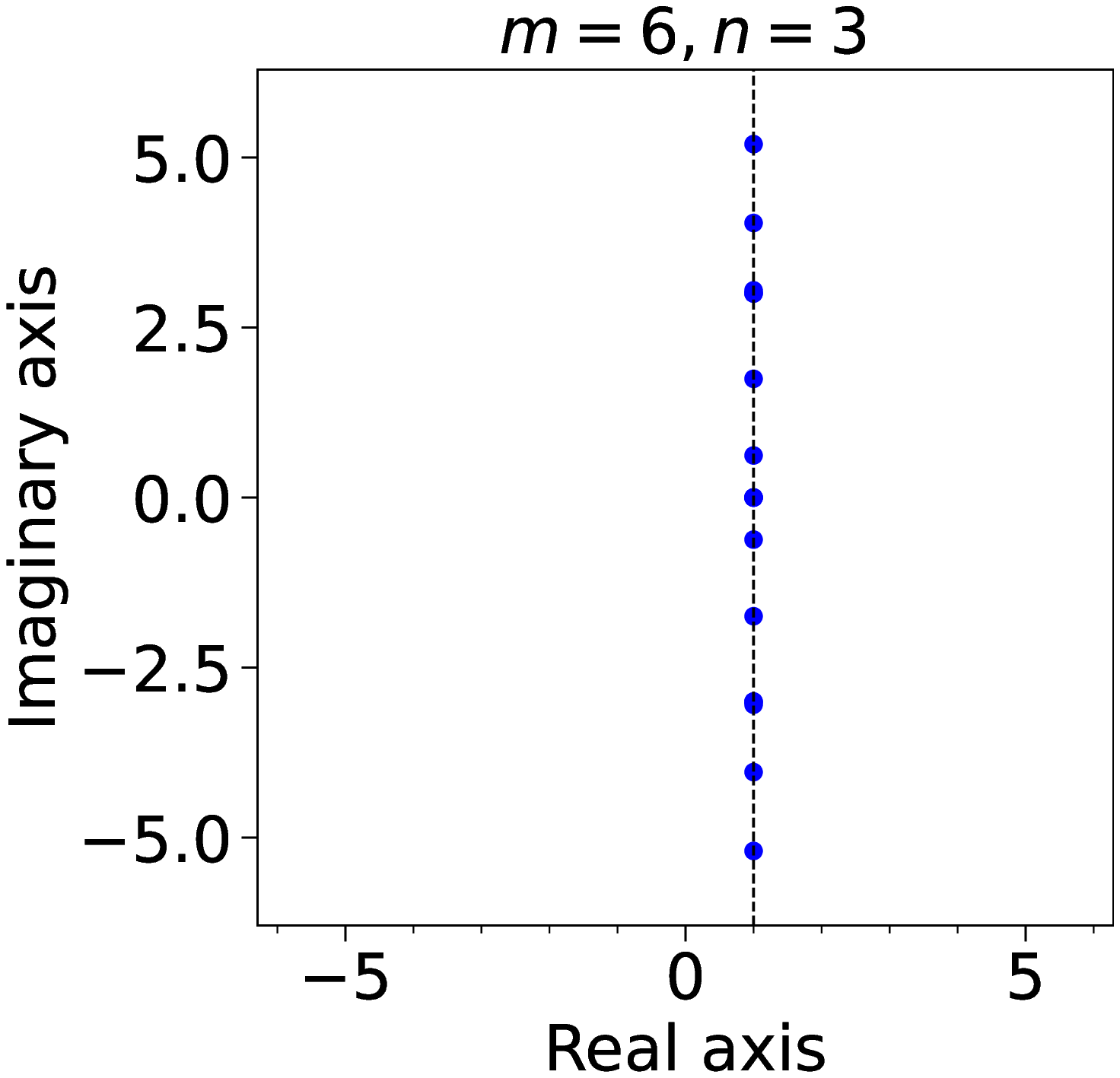}
          \label{fig:embr_6_3}
        \end{minipage} \\
        \begin{minipage}[t]{0.3\linewidth}
          \centering
          \includegraphics[scale=0.3]{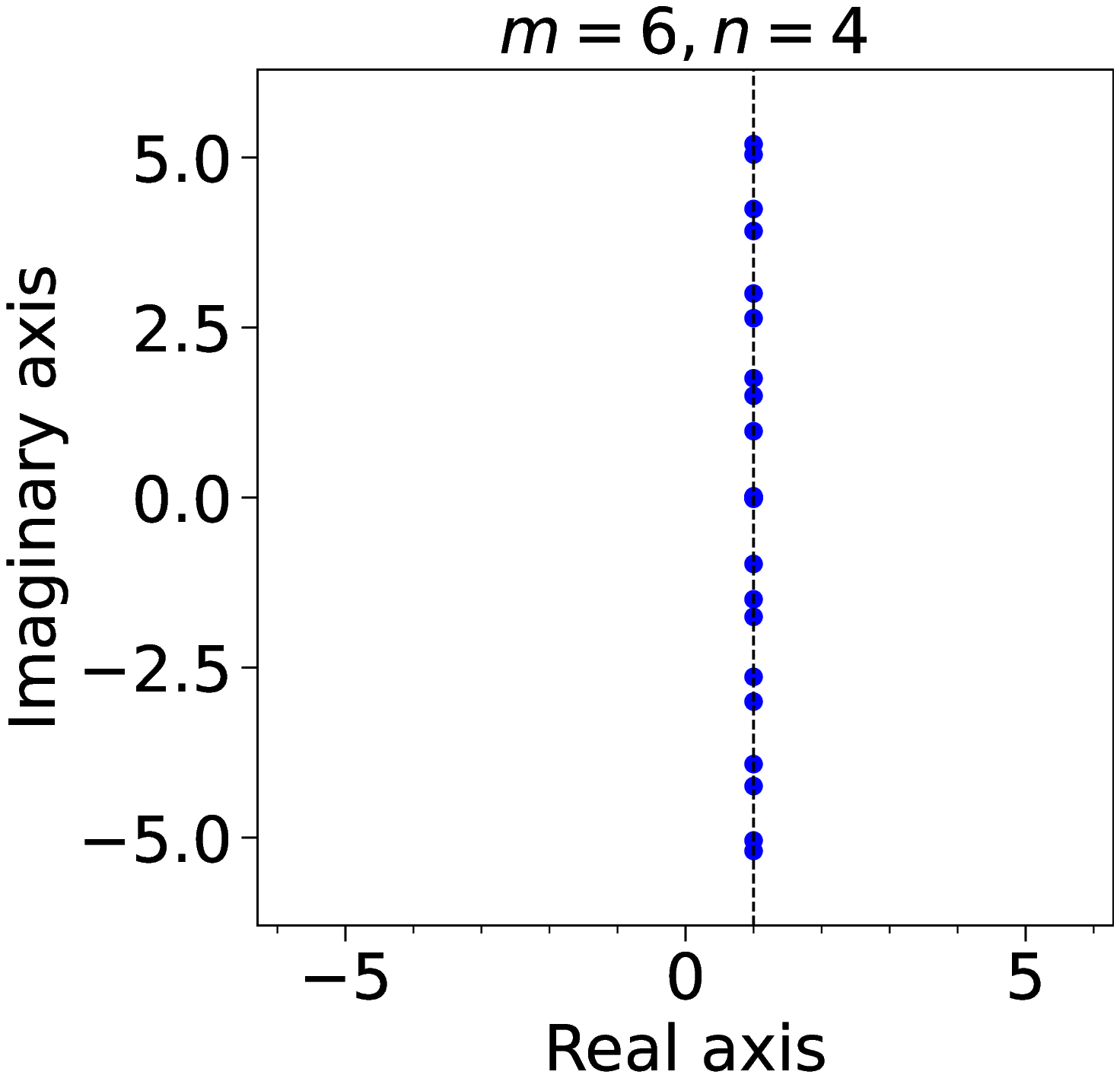}
          \label{fig:embr_6_4}
        \end{minipage} &
        \begin{minipage}[t]{0.3\linewidth}
          \centering
          \includegraphics[scale=0.3]{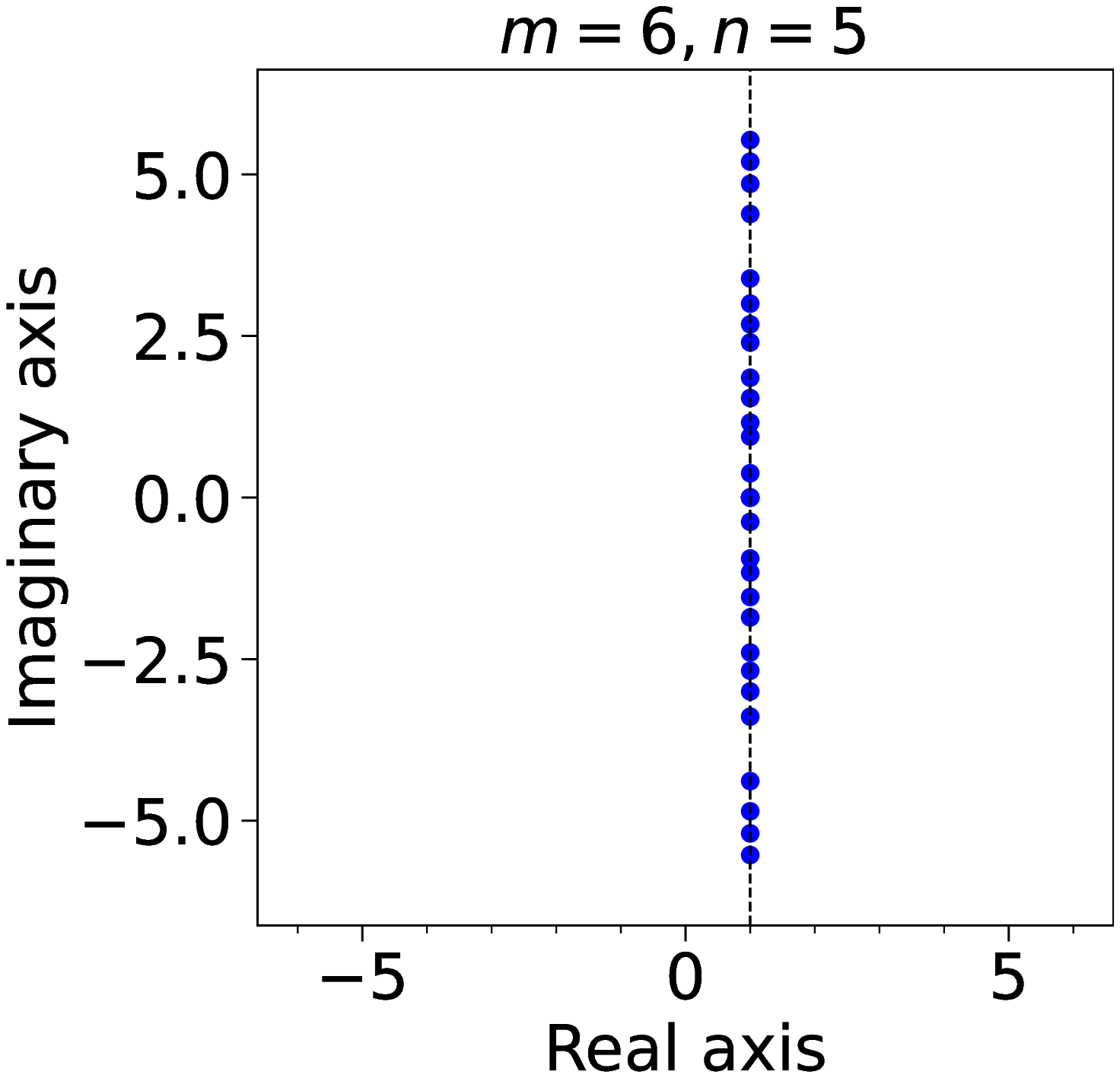}
          \label{fig:embr_6_5}
        \end{minipage} & 
        \begin{minipage}[t]{0.3\linewidth}
          \centering
          \includegraphics[scale=0.3]{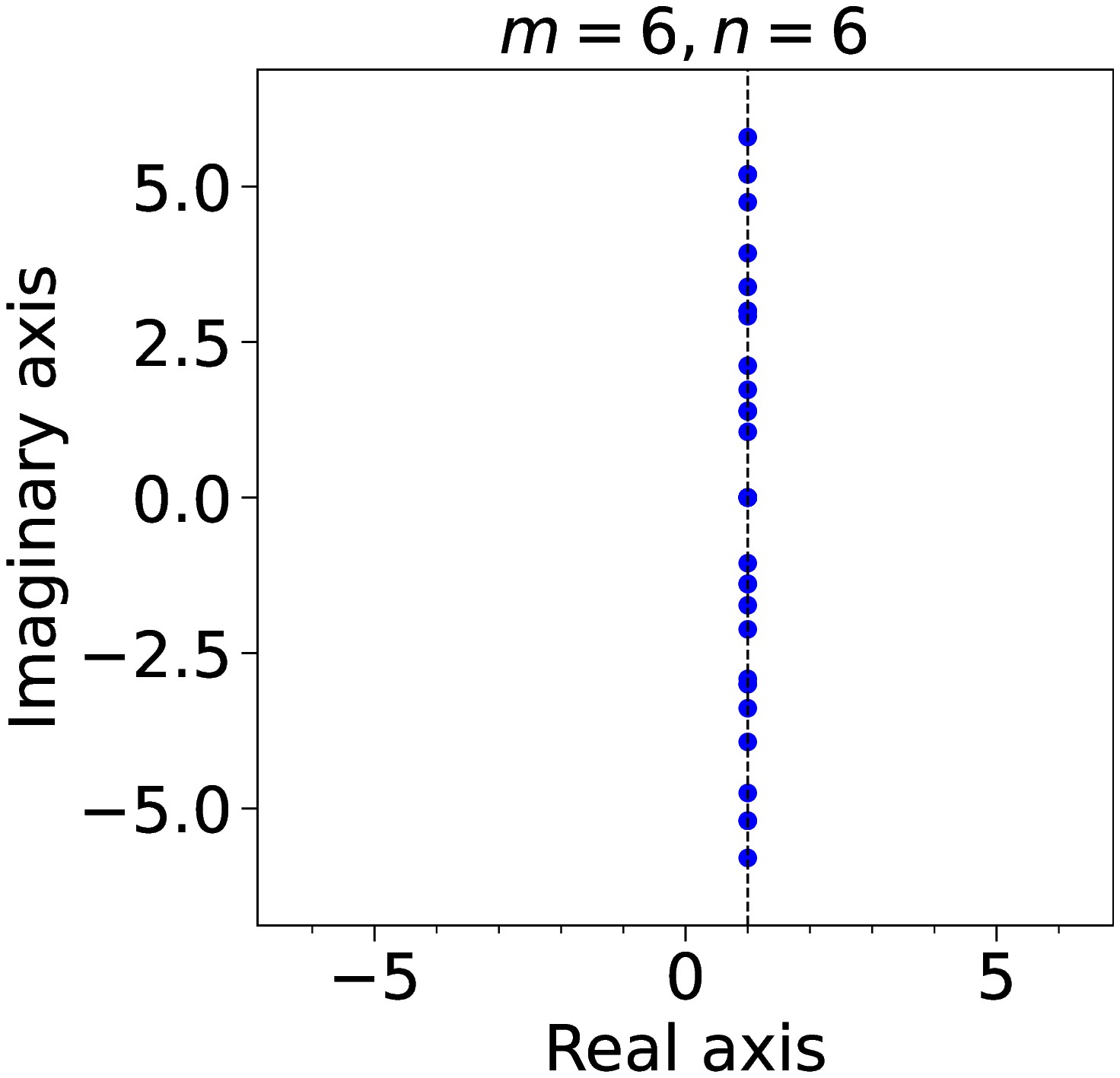}
          \label{fig:embr_6_6}
        \end{minipage} 
      \end{tabular}
      \caption{Eigenvalues of $\mathcal{F}$~\eqref{coef:embr}. 
      The dashed line represents $\mathrm{Re}(z) = 1, z\in\mathbb{C}$.
      }
      \label{fig:eigdist}
    \end{figure}

\section{Conclusion\label{sec4}}
  In this paper, it was shown that the convolution equation~\eqref{eq:convolution} under the zero, periodic, or reflexive boundary condition can be represented as a generalized Sylvester equation.
  Concretely, the result was obtained by using the shift matrices~\eqref{shiftmatrix}, the cyclic shift matrices~\eqref{cshiftmatrix}, and the tridiagonal matrices defined by~\eqref{qshiftmatrix}, for the zero, periodic, and reflexive boundary conditions, respectively.
  In addition, for some concrete examples arising from image processing, we showed that the generalized Sylvester equation that is equivalent to the convolution equation can be reduced to a simpler form, and we have characterized the unique solvability of the convolution equation.
  The necessary and sufficient conditions for the convolution equation~\eqref{eq:convolution} to have a unique solution for every right-hand side $B$ are summarized in Table~\ref{table}.

  In the future, we will consider finding numerical algorithms for the convolution equation~\eqref{eq:convolution} by utilizing the representation of the convolution equation as a generalized Sylvester equation.
  Proving the conjecture in Subsubsection~\ref{subsubsec3.7.3} is one of our future work.

  \begin{table}[hbtp]
    \caption{The necessary and sufficient conditions for the unique solvability of the convolution equation~\eqref{eq:convolution}.}\label{table}
    \centering
    \begin{tabular}{c||c|c|c}
      \hline
      & zero & periodic & reflexive\\
      \hline\hline
      BOX & $m,n\notin \{3l-1:l\in\mathbb{N}\}$ & $m,n\notin \{3l:l\in\mathbb{N}\}$ & $m,n\notin \{3l:l\in\mathbb{N}\}$ \\
      \hline
      GUS & $m,n\in \mathbb{N}$ & $m,n\notin \{2l:l\in\mathbb{N}\}$ & $m,n\in \mathbb{N}$ \\
      \hline
      EDA & $m,n\notin \{2l-1:l\in\mathbb{N}\}$ & no unique solution & no unique solution \\
      \hline
      EDB & $m,n\in \mathbb{N}$ & no unique solution & no unique solution \\
      \hline
      EDC & $m,n\in \mathbb{N}$ & no unique solution & no unique solution \\
      \hline
      SHP & $m,n\in \mathbb{N}$ & $m,n\in \mathbb{N}$ & $m,n\in \mathbb{N}$ \\
      \hline
      EMB & $m,n\in \mathbb{N}$ & $m,n\in \mathbb{N}$ & $m,n\in \mathbb{N}$ (conjecture) \\
      \hline
    \end{tabular}
  \end{table}

\section*{Acknowledgments}
  The authors appreciate Prof. Carlos Martins da Fonseca for informing us of the existing research on Lemma~\ref{lmm:qteig}.
  This work is supported by JSPS KAKENHI Grant number 21J15734.

\bibliographystyle{jabbrv}

\section*{Appendix: The Proof of Lemma~\ref{lmm:qteig}}
  Here we provide a proof of Lemma~\ref{lmm:qteig}.
  \begin{proof}
  We show that the eigenvalues $\lambda_1,\ldots,\lambda_n$ of
  \begin{equation}
    A=\mqty{
      \alpha + \gamma & \beta  &        &        & \\
      \gamma          & \alpha & \beta  &        & \\
                      & \gamma & \ddots & \ddots & \\
                      &        & \ddots & \alpha & \beta \\
                      &        &        & \gamma & \alpha + \beta
    }\in\mathbb{R}^{n\times n} \label{A_qtap}
  \end{equation}
  can be written by
  \begin{align}
    \begin{aligned}
      \lambda_k=
      \begin{cases}
        \alpha + 2\sqrt{\beta\gamma}\cos\left(-\frac{k\pi}{n}\right), &\qquad k=1,2,\ldots,n-1,\\
        \alpha + \beta + \gamma, &\qquad k=n.
      \end{cases}
    \end{aligned}\label{eig_pr}
  \end{align}
  Let $f_n(\lambda):=\det(A-\lambda I_n)$ and $\tilde{f}_n(\lambda):=\det(\tilde{A}-\lambda I_n)$ where
  \begin{equation}
    \tilde{A}=\mqty{
      \alpha & \beta  & & \\
      \gamma & \alpha & \ddots & \\
             & \ddots & \ddots & \beta \\
             &        & \gamma & \alpha
    }\in\mathbb{R}^{n\times n}.
  \end{equation}
  Then, by cofactor expansion, 
  \begin{align}
    f_n(\lambda)=&(\alpha -\lambda + \gamma)\mdet{
      \alpha - \lambda & \beta            &        & \\
      \gamma           & \ddots & \ddots  &           \\
                       & \ddots & \alpha - \lambda & \beta \\
                       &        & \gamma           & \alpha - \lambda + \beta      
    }_{n-1} - \beta\mdet{
      \gamma & \beta            &        &                  & \\
             & \alpha - \lambda & \beta  &                  & \\
             & \gamma           & \ddots & \ddots           & \\
             &                  & \ddots & \alpha - \lambda & \beta \\
             &                  &        & \gamma           & \alpha - \lambda + \beta      
    }_{n-1} \\
    =&(\alpha -\lambda + \gamma)\left[(\alpha - \lambda + \beta)\tilde{f}_{n-2} - \gamma\mdet{
      \alpha - \lambda & \beta            &        &                          & \\
      \gamma           & \ddots & \ddots &                          & \\
                       & \gamma           & \ddots & \beta                    & \\
                       &                  & \ddots & \alpha - \lambda & \\
                       &                  &        & \gamma & \beta
    }_{n-2}\right]
    \\ & \quad
     - \beta\gamma\mdet{
      \alpha - \lambda & \beta  &                  & \\
      \gamma           & \ddots & \ddots           & \\
                       & \ddots & \alpha - \lambda & \beta \\
                       &        & \gamma           & \alpha - \lambda + \beta      
    }_{n-2} \\
    =&(\alpha -\lambda + \gamma)\left[(\alpha - \lambda + \beta)\tilde{f}_{n-2} - \beta\gamma \tilde{f}_{n-3}\right] - \beta\gamma\left[(\alpha - \gamma + \beta)\tilde{f}_{n-3} - \gamma\mdet{
      \alpha - \lambda & \beta  &        & & \\
      \gamma           & \ddots & \ddots & & \\
                       & \ddots & \ddots & \beta &  \\
                       &        & \ddots & \alpha - \lambda & \\
                       &        &        & \gamma & \beta
    }_{n-3}\right] \\
    =& (\alpha -\lambda + \gamma)\left[(\alpha - \lambda + \beta)\tilde{f}_{n-2} - \beta\gamma \tilde{f}_{n-3}\right] - \beta\gamma\left[(\alpha - \lambda + \beta)\tilde{f}_{n-3} - \beta\gamma\tilde{f}_{n-4}\right] \\
    =& (\alpha -\lambda + \gamma)(\alpha - \lambda + \beta)\tilde{f}_{n-2}
    - \left[2(\alpha - \lambda) + \beta + \gamma \right] \beta\gamma\tilde{f}_{n-3}
    + (\beta\gamma)^2\tilde{f}_{n-4},\label{cofexp}
  \end{align}
  where $\left|\cdot\right|_n$ represents a determinant of $n\times n$ matrix.
  By cofactor expansion, we have the following recurrence relation:
  \begin{equation}
    \tilde{f}_n = (\alpha - \lambda)\tilde{f}_{n-1} - \beta\gamma\tilde{f}_{n-2},
  \end{equation}
  which implies
  \begin{align}
    \beta\gamma\tilde{f}_{n-3} &= (\alpha - \lambda)\tilde{f}_{n-2} - \tilde{f}_{n-1},\label{ss1}\\
    (\beta\gamma)^2\tilde{f}_{n-4} &= (\alpha - \lambda)\beta\gamma\tilde{f}_{n-3} - \beta\gamma\tilde{f}_{n-2} = -(\alpha - \lambda)\tilde{f}_{n-1} + \left[(\alpha - \lambda)^2- \beta\gamma\right]\tilde{f}_{n-2}.\label{ss2}
  \end{align}
  By substituting~\eqref{ss1} and~\eqref{ss2} into~\eqref{cofexp}, we obtain
  \begin{align}
    f_n(\lambda) =& (\alpha -\lambda + \gamma)(\alpha - \lambda + \beta)\tilde{f}_{n-2}
    - \left[2(\alpha - \lambda) + \beta + \gamma \right] \left[(\alpha - \lambda)\tilde{f}_{n-2} - \tilde{f}_{n-1}\right]
    \\ & \quad 
    -(\alpha - \lambda)\tilde{f}_{n-1} + \left[(\alpha - \lambda)^2- \beta\gamma\right]\tilde{f}_{n-2} \\
    =& \left[(\alpha - \lambda)^2 + (\beta + \gamma)(\alpha - \lambda) + \beta\gamma - 2(\alpha - \lambda)^2 - (\beta + \gamma)(\alpha - \lambda) + (\alpha - \lambda)^2 - \beta\gamma\right]\tilde{f}_{n-2}
    \\ & \quad
     + (\alpha + \beta +\gamma - \lambda)\tilde{f}_{n-1} \\
     =& (\alpha + \beta +\gamma - \lambda)\tilde{f}_{n-1}.
  \end{align}
  Thus, the eigenvalues of~\eqref{A_qtap} are $\lambda_n = \alpha + \beta + \gamma$ and roots of $\tilde{f}_{n-1}(\lambda)$.
  Since $\tilde{f}_{n-1}$ is the charactaristic polynomial of $(n-1)\times (n-1)$ tridiagonal Toeplitz matrix, 
  its roots are given by
  \begin{equation}
    \lambda_k = \alpha + 2\sqrt{\beta\gamma}\cos\left(\frac{k\pi}{n}\right),\qquad k=1,2,\ldots, n-1.
  \end{equation}
  Hence, the proof is complete.
  \end{proof}

\end{document}